\documentclass[12pt]{amsart}
\usepackage[margin=1.5in]{geometry}
\usepackage[OT2,T1]{fontenc}
\usepackage[utf8]{inputenc}
\usepackage{tikz-cd}
\usepackage{cancel}
\usepackage{stmaryrd}
\SetSymbolFont{stmry}{bold}{U}{stmry}{m}{n}
\usepackage{microtype,hyperref,amsmath,amsfonts,amssymb,amsthm,mathrsfs,amsrefs,bm}
\usepackage[scr=rsfs,cal=cm]{mathalfa}

\theoremstyle{plain}
\newtheorem{theorem}[]{Theorem}
\newtheorem{proposition}[theorem]{Proposition}
\newtheorem{lemma}[theorem]{Lemma}

\newtheorem{corollary}[theorem]{Corollary}

\theoremstyle{definition}
\newtheorem{definition}[]{Definition}

\newtheorem{remark}[definition]{Remark}

\usepackage{enumitem}
\numberwithin{theorem}{section}
\numberwithin{definition}{section}
\numberwithin{equation}{section}
\setlist[itemize]{leftmargin=2em}

\title{On Calabi--Yau fractional complete intersections}
\date{\today}

\author[T.-J.~Lee]{Tsung-Ju~Lee}
\address{Tsung-Ju~Lee: Center of Mathematical Sciences and Applications, 20 Garden St., Cambridge, MA 02138, U.S.A.}
\email{tjlee@cmsa.fas.harvard.edu}

\author[B.~Lian]{Bong~H.~Lian}
\address{Bong~H.~Lian, Department of Mathematics, Brandeis University, Waltham MA 02454, U.S.A.}
\email{lian@brandeis.edu}

\author[S.-T.~Yau]{Shing-Tung~Yau}
\address{Shing-Tung~Yau, Department of Mathematics, Harvard University, Cambridge MA 02138, U.S.A.~\& Yau Mathematical Sciences Center, Tsinghua University, Beijing 100084, China}
\email{yau@math.harvard.edu}

\begin{document}
\maketitle

\begin{center}
\textit{in Honor of Bernard Shiffman's 75th birthday.}
\end{center}

\begin{abstract}
In this article, we study mirror symmetry for pairs of \emph{singular} 
Calabi--Yau varieties which are double covers of
toric manifolds. Their period integrals can be seen as 
certain `fractional' analogues of those of ordinary complete intersections.
This new structure can then be used to solve their Riemann--Hilbert problems. 
The latter can then be used to answer
definitively questions about mirror symmetry for this class of Calabi--Yau varieties. 
\end{abstract}

\tableofcontents
\section{Introduction}
\subsection{Motivation}
Mirror symmetry from physics has successfully made numerous predictions in 
algebraic geometry and attracted lots of attentions in the past thirty years.
Roughly, mirror symmetry asserts that for a Calabi--Yau space \(X\) there
exists a Calabi--Yau space \(X^{\vee}\) such that \(A(X)\cong B(X^{\vee})\) and
\(B(X)\cong A(X^{\vee})\). Here \(A(X)\), the \(A\) model of \(X\), 
is taken to be the genus zero Gromov--Witten theory whereas 
\(B(X)\), the \(B\) model of \(X\), is the variation of Hodge structures. 

Various examples of mirror pairs have been constructed. 
The first mirror pair was given by Greene and Plesser
\cite{1990-Greene-Plesser-duality-in-calabi-yau-moduli-space}, leading to the spectacular prediction of genus zero Gromov--Witten invariants for quintic threefolds
\cite{1991-Candelas-de-la-Ossa-Green-Parkes-a-pair-of-calabi-yau-manifolds-as-an-exactly-soluable-superconformal-theory}.
Batyrev generalized
the construction to the case of Calabi--Yau hypersurfaces in Gorenstein Fano toric varieties 
by making use of reflexive polytopes \cite{1994-Batyrev-dual-polyhedra-and-mirror-symmetry-for-calabi-yau-hypersurfaces-in-toric-varieties}, 
leading to similar predictions of Gromov--Witten invariants for general Calabi--Yau toric hypersurfaces 
\cite{1996-Hosono-Lian-Yau-gkz-generalized-hypergeometric-systems-in-mirror-symmetry-of-calabi-yau-hypersurfaces}.
Later, Batyrev and Borisov gave a general recipe to construct mirror pairs 
in the case of Calabi--Yau complete intersections in Gorenstein Fano toric varieties 
by nef-partitions \cite{1996-Batyrev-Borisov-on-calabi-yau-complete-intersections-in-toric-varieties}. At the same time, Bershadsky et~al.~\cite{1994-Bershadsky-Cecotti-Ooguri-Vafa-kodaira-spencer-theory-of-gravity-and-exact-results-for-quantum-string-amplitudes}
developed their fair reaching theory of topological strings which led to predictions of 
Gromov--Witten invariants in higher genera.

Thanks to Torelli theorem, 
the \(B\) model is locally completely determined by period integrals of the
Calabi--Yau families. It is known that the period integrals satisfy a set of partial 
differential equations, known as the Picard--Fuchs equations. 
Batyrev observed in \cite{1993-Batyrev-variations-of-the-mixed-hodge-structure-of-affine-hypersurfaces-in-algebraic-tori} that the period integrals
of a family of Calabi--Yau hypersurfaces or complete intersections in
a fixed Gorenstein Fano toric variety satisfy a generalized hypergeometric system 
introduced by Gel'fand, Kapranov, and Zelevinski\u{\i} 
\cite{1989-Gelfand-Zelevinskii-Kapranov-hypergeometric-functions-and-toric-varieties},
which is called the GKZ \(A\)-hypergeometric system nowadays. 
For a family of Calabi--Yau hypersurfaces or complete intersections in toric varieties, 
we attempt to understand its period integrals through the GKZ \(A\)-hypergeometric systems 
associated with it.

Hosono et~al.~observed that the
Gr\"{o}bner basis with respect to 
the typical weight for the toric ideal determines a finite set of differential operators
for the local solutions to the GKZ \(A\)-hypergeometric system
\cite{1996-Hosono-Lian-Yau-gkz-generalized-hypergeometric-systems-in-mirror-symmetry-of-calabi-yau-hypersurfaces}.
For such a GKZ \(A\)-hypergemetric system, 
they also proved that there exists a special boundary point called a maximal degeneracy point 
on a resolution of the secondary fan compactication of the moduli 
\cite{1997-Hosono-Lian-Yau-maximal-degeneracy-points-of-gkz-systems}.
It is a point over which for all but one period integrals can not be extended holomorphically;
namely, up to a constant, there exists a unique holomorphic period at that point.
To study the moduli locally near such a special boundary point, 
inspired by mirror symmetry,
the \emph{generalized Frobenius method} was developed in 
\cites{1995-Hosono-Klemm-Theisen-Yau-mirror-symmetry-mirror-map-and-applications-to-complete-intersection-calabi-yau-spaces,1996-Hosono-Lian-Yau-gkz-generalized-hypergeometric-systems-in-mirror-symmetry-of-calabi-yau-hypersurfaces}.
Starting with the holomorphic period, the method allows us to 
produce other period integrals.
The generalized Frobenius method gives a uniform treatment
to describe the local solutions near a maximal degeneracy point in the moduli.

The works in
\cites{2020-Hosono-Lian-Takagi-Yau-k3-surfaces-from-configurations-of-six-lines-in-p2-and-mirror-symmetry-i,2019-Hosono-Lian-Yau-k3-surfaces-from-configurations-of-six-lines-in-p2-and-mirror-symmetry-ii-lambda-k3-functions} shed light on a new construction of a mirror pair of
\emph{singular} Calabi--Yau varieties. 
Hosono, Takagi and the last two authors investigated
the family of \(K3\) surfaces arising from double covers branched over 
six lines in \(\mathbb{P}^{2}\) and proposed a singular version
of mirror symmetry. 
Recently, together with Hosono, the authors 
gave a general recipe to construct pairs of \emph{singular} Calabi--Yau varieties 
\((Y,Y^{\vee})\)
and showed that they are topological mirror pairs in dimension three
\cite{2020-Hosono-Lee-Lian-Yau-mirror-symmetry-for-double-cover-calabi-yau-varieties};
in other words, we have \(h^{p,q}(Y) = h^{3-p,q}(Y^{\vee})\) for all
\(0\le p,q\le 3\).

\subsection{Statements of main results}
The aim of this note is to straighten the results in 
\cites{1996-Hosono-Lian-Yau-gkz-generalized-hypergeometric-systems-in-mirror-symmetry-of-calabi-yau-hypersurfaces,1997-Hosono-Lian-Yau-maximal-degeneracy-points-of-gkz-systems} to 
our singular topological mirror pairs.

Consider a nef-partition \((\Delta,\{\Delta_{i}\}_{i=1}^{r})\)
and its dual nef-partition \((\nabla,\{\nabla_{i}\}_{i=1}^{r})\) in the sense of
Batyrev and Borisov.
Let \(\mathbf{P}_{\Delta}\) and \(\mathbf{P}_{\nabla}\)
be the toric varieties defined by \(\Delta\) and \(\nabla\). Let
\(X\to\mathbf{P}_{\Delta}\) and \(X^{\vee}\to \mathbf{P}_{\nabla}\) be
maximal projective crepant partial resolutions (MPCP resolutions for short hereafter)
of \(\mathbf{P}_{\Delta}\) and \(\mathbf{P}_{\nabla}\).
The nef-partitions on \(\mathbf{P}_{\Delta}\) and \(\mathbf{P}_{\nabla}\)
determine nef-partitions on \(X\) and \(X^{\vee}\). Let 
\(E_{1},\ldots,E_{r}\) and \(F_{1},\ldots,F_{r}\) be the sum of
toric divisors representing nef-partitions on \(X\) and \(X^{\vee}\).
We assume that both
\(X\) and \(X^{\vee}\) are \emph{smooth} throughout this note.

For a nef-partition \(F_{1}+\cdots+F_{r}\) on \(X^{\vee}\),
we can define a family \(\mathcal{Y}^{\vee}\to V\) of singular Calabi--Yau varieties as follows.
For each \(i\), let \(s_{i,1},s_{i,2}\in\mathrm{H}^{0}(X^{\vee},F_{i})\) be sections such that
\(\mathrm{div}(s_{i,1})\equiv F_{i}\) and \(\mathrm{div}(s_{i,2})\) is smooth.
Let \(Y^{\vee}\to X^{\vee}\) be a double cover branched over
\(\cup_{i=1}^{r}\cup_{j=1}^{2}\mathrm{div}(s_{i,j})\). 
Let 
\begin{equation*}
V\subset W^{\vee}:=\mathrm{H}^{0}(X^{\vee},F_{1})
\times\cdots\times\mathrm{H}^{0}(X^{\vee},F_{r})
\end{equation*}
be an open subset such that \(\sum_{i=1}^{r}\sum_{j=1}^{2}\mathrm{div}(s_{i,j})\)
is a simple normal crossing divisor. Deforming \(s_{i,2}\) in \(V\), we obtain the said
family of double covers of \(X\), which is called
\emph{the gauge fixed double covers family} in this paper.
Similarly, the dual nef partition \(E_{1}+\cdots+E_{r}\) gives 
another family \(\mathcal{Y}\to U\).

To state our main results, let us introduce some notation.
Let \(N\simeq\mathbb{Z}^{n}\) 
be a lattice in which the fan of \(X\) sits and 
\(M:=\mathrm{Hom}_{\mathbb{Z}}(N,\mathbb{Z})\). 
Let \(\Sigma\) be the fan defining \(X\). The nef-partition 
\(E_{1}+\cdots+E_{r}\) on \(X\) determines
a decomposition \(\sqcup_{k=1}^{r} I_{k}\) of 
\(\Sigma(1)\), the set of \(1\)-cones in \(\Sigma\).
We can write
\begin{equation*}
\Sigma(1)=\left\{\rho_{i,j}\colon 
\rho_{i,j}\in I_{i}~\mbox{for}~
~1\le i\le r,~1\le j\le n_{i}=\#I_{i}
\right\}.
\end{equation*}
The primitive generator for the \(1\)-cone \(\rho_{i,j}\) is again denoted by \(\rho_{i,j}\).
For \(1\le i\le r\) and \(1\le j\le n_{i}\),
we put \(\nu_{i,j}:=(\rho_{i,j},\delta_{1,i},\ldots,\delta_{r,i})\) 
and additionally
\(\nu_{i,0}:=(\mathbf{0},\delta_{1,i},\ldots,\delta_{r,i})\),
where \(\delta_{i,j}\) is the Kronecker delta.
Let 
\begin{equation*}
A_{\mathrm{ext}}:=
\begin{bmatrix}
\nu_{1,0}^{\intercal} & \cdots & \nu_{r,n_{r}}^{\intercal}
\end{bmatrix}\in\mathrm{Mat}_{(n+r)\times(p+r)}(\mathbb{Z}),~
p=n_{1}+\cdots+n_{r}.
\end{equation*}
It turns out that the \emph{affine period integrals} 
(For a precise definition, 
see \S\ref{subsection:affine-period-integrals}.) 
\begin{equation}
\Pi_{\gamma}(x):=\int_{\gamma} 
\frac{1}{s_{1,2}^{1/2}\cdots s_{r,2}^{1/2}}\frac{\mathrm{d}t_{1}}{t_{1}}\wedge\cdots
\wedge\frac{\mathrm{d}t_{n}}{t_{n}}
\end{equation}
for the gauge fixed double cover family 
\(\mathcal{Y}^{\vee}\to V\) 
satisfy a GKZ \(A\)-hypergeometric system associated 
with the matrix \(A_{\mathrm{ext}}\) and 
a \emph{fractional} exponent 
\begin{equation*}
\beta = \begin{bmatrix}\mathbf{0} & -1/2 & \cdots & -1/2\end{bmatrix}^{\intercal}
\in\mathbb{Q}^{n+r}.
\end{equation*}
Note that for ordinary complete intersections the 
exponents appearing in the denominator in the affine period integrals 
would be integers. But for gauge fixed double cover families, 
the exponents become half integers 
(Hence `fractional' complete intersections). 

The affine period integrals of \(\mathcal{Y}^{\vee}\to V\) form a 
local system on \(W^{\vee}\setminus\mathcal{D}\) for some closed subset \(\mathcal{D}\).
Let \(T_{M}:=\mathrm{Hom}(N,\mathbb{C}^{\ast})\). 
The space \(W^{\vee}\) is equipped with a \(T_{M}\times(\mathbb{C}^{\ast})^{r}\) action 
via the inclusion 
\(T_{M}\times(\mathbb{C}^{\ast})^{r}\hookrightarrow(\mathbb{C}^{\ast})^{\dim W^{\vee}}\) and 
the affine periods are invariant under this action. In other words, the periods descend
to local sections of a locally constant sheaf on \(S_{W^{\vee}}\), where \(S_{W^{\vee}}\)
is the image of \((\mathbb{C}^{\ast})^{\dim W^{\vee}}\setminus \mathcal{D}\) under
\begin{equation*}
(\mathbb{C}^{\ast})^{\dim W^{\vee}}\to (\mathbb{C}^{\ast})^{\dim W^{\vee}}\slash 
T_{M}\times(\mathbb{C}^{\ast})^{r}.
\end{equation*}
Following the idea in \cite{1997-Hosono-Lian-Yau-maximal-degeneracy-points-of-gkz-systems},
we compactify \((\mathbb{C}^{\ast})^{\dim W^{\vee}}\slash 
T_{M}\times(\mathbb{C}^{\ast})^{r}\) into a toric variety 
via the secondary fan \(S\Sigma\) and the Gr\"{o}bner fan \(G\Sigma\). Our first 
theorem in this note is
\begin{theorem}[=Theorem \ref{thm:main-theorem-1}]
\label{thm:main-thm-1-intro}
For every toric resolution \(X_{G\Sigma'}\to X_{G\Sigma}\), there exists
at least one maximal degeneracy point in \(X_{G\Sigma'}\). 
\end{theorem}
The precise definition of maximal degeneracy points is given in Definition 
\ref{def:maximal-degeneracy-points}.
The secondary fan \(S\Sigma\) is natural from combinatorics whereas
the Gr\"{o}bner fan \(G\Sigma\) contains more information about our GKZ system.
The proof of Theorem \ref{thm:main-thm-1-intro} is parallel to the proof given
in \cite{1997-Hosono-Lian-Yau-maximal-degeneracy-points-of-gkz-systems}.

Let \(L_{\mathrm{ext}}:=\mathrm{ker}
(A_{\mathrm{ext}}\colon \mathbb{Z}^{p+r}\to\mathbb{Z}^{n+r})\).
Note that the Mori cone \(\overline{\mathrm{NE}}(X)\) 
is a cone in \(L_{\mathrm{ext}}\otimes\mathbb{R}\).
Pick an \(\alpha\in\mathbb{C}^{p+r}\) such that \(A_{\mathrm{ext}}(\alpha)=\beta\).
As observed in \cite{1996-Hosono-Lian-Yau-gkz-generalized-hypergeometric-systems-in-mirror-symmetry-of-calabi-yau-hypersurfaces},
after a renormalization, a 
solution to the GKZ system is given by
\begin{equation}
\label{eq:series-solution-intro}
\sum_{\ell\in L_{\mathrm{ext}}} 
\frac{\prod_{i=1}^{r}\Gamma(-\ell_{i,0}-\alpha_{i,0})}
{\prod_{i=1}^{r}\Gamma(-\alpha_{i,0})
\prod_{i=1}^{r}\prod_{j=1}^{n_i} \Gamma(\ell_{i,j}+\alpha_{i,j}+1)}
(-1)^{\sum_{i} \ell_{i,0}}x^{\ell+\alpha}.
\end{equation}
Here the components of \(L_{\mathrm{ext}}\subset \mathbb{Z}^{p+r}\) are labeled by \((i,j)\)
with \(1\le i\le r\) and \(0\le j\le n_{i}\). The variables \(x_{i,j}\) 
(again \(1\le i\le r\) and \(0\le j\le n_{i}\)) are the 
coordinates for the GKZ \(A\)-hypergeometric system associated 
with \(\mathcal{Y}^{\vee}\to V\).

Let \(D_{i,j}\) be the Weil divisor associated with \(\rho_{i,j}\). Combining 
\ref{eq:series-solution-intro} with
these cohomology classes, we introduce a cohomology-valued power series
\begin{equation}
\label{eq:coh-valused-series}
B_{X}^{\alpha}(x):=\left(\sum_{\ell\in \overline{\mathrm{NE}}(X)\cap L_{\mathrm{ext}}} 
\mathcal{O}_{\ell}^{\alpha} x^{\ell+\alpha}\right)
\exp\left(\sum_{i=1}^{r}\sum_{j=0}^{n_i}(\log x_{i,j}) D_{i,j}\right),
\end{equation}
where 
\begin{equation*}
\mathcal{O}^\alpha_\ell
:=\frac{\prod_{i=1}^{r}(-1)^{\ell_{i,0}}
\Gamma(-D_{i,0}-\ell_{i,0}-\alpha_{i,0})}{\prod_{i=1}^{r}\Gamma(-\alpha_{i,0})
\prod_{i=1}^r\prod_{j=1}^{n_i}\Gamma(D_{i,j}+\ell_{i,j}+\alpha_{i,j}+1)}.
\end{equation*}
and \(D_{i,0}:=-\sum_{j=1}^{n_{i}} D_{i,j}\). 

The cohomology-valued series \eqref{eq:coh-valused-series} 
was introduced by Hosono et~al.~in
\cite{1996-Hosono-Lian-Yau-gkz-generalized-hypergeometric-systems-in-mirror-symmetry-of-calabi-yau-hypersurfaces} 
(a.k.a.~Givental's \(I\)-function up to
an overall \(\Gamma\)-factor
\cite{1998-Givental-a-mirror-theorem-for-toric-complete-intersections})
which encodes the information from the \(A\) model and the \(B\) model for 
a Calabi--Yau mirror pair. 

We regard \(B^{\alpha}_{X}(x)\) as an element in 
\(\mathbb{C}\llbracket x_{i,j}\rrbracket
\otimes_{\mathbb{C}}\mathrm{H}^{\bullet}(X,\mathbb{C})\).
Our second result in this note is the following theorem.
\begin{theorem}[=Corollary \ref{cor:full-set-soultion}]
\label{prop:coh-ser-solu-gkz-intro}
When \(h\in\mathrm{H}^{\bullet}(X,\mathbb{C})^{\vee}\)
runs through a basis of \(\mathrm{H}^{\bullet}(X,\mathbb{C})^{\vee}\),
the pairings \(\langle B^{\alpha}_{X}(x),h\rangle\)
give a complete set of solution to the GKZ \(A\)-hypergeometric system associated 
with \(\mathcal{Y}^{\vee}\to V\). 
\end{theorem}

A direct calculation shows that \(\langle B^{\alpha}_{X}(x),h\rangle\)
is a solution to the GKZ \(A\)-hypergeometric system associated 
with \(\mathcal{Y}^{\vee}\to V\).
See also \cite{2006-Borisov-Horja-mellin-barnes-integrals-as-fourier-mukai-transforms}.
The dimension of the solution space to this GKZ system
is given by the normalized volume of \(A_{\mathrm{ext}}\),
which turns out to be equal to the dimension of \(\mathrm{H}_{n}(Y^{\vee},\mathbb{C})\)
if \(n\) is odd for a generic fiber \(Y^{\vee}\).

Theorem \ref{prop:coh-ser-solu-gkz-intro}
solves the Riemann--Hilbert problem for the periods of the family of Calabi--Yau varieties $\mathcal{Y}^{\vee }$. It gives a complete description for the Picard--Fuchs system of the periods of this family in terms of a GKZ system.

\subsection*{Acknowledgment} The work presented here is based on joint works
with Shinobu Hosono. We thank him for invaluable discussions.
We thank Center of Mathematical Sciences and Applications
at Harvard for hospitality while working on this project. 
We also thank anonymous referees for reading our manuscript carefully and providing 
useful comments.
B.~H.~Lian and S.-T.~Yau are supported by the Simons Collaboration Grant on Homological Mirror Symmetry and Applications 2015-2022.

\section{Preliminaries}
\label{subsection:notation-all}
We begin with some notation and terminologies.
\begin{itemize}
\item Let \(N=\mathbb{Z}^{n}\) be a rank \(n\) lattice and
\(M=\mathrm{Hom}_{\mathbb{Z}}(N,\mathbb{Z})\) be its dual lattice. 
Let \(N_{\mathbb{R}}:=N\otimes_{\mathbb{Z}}\mathbb{R}\) and 
\(M_{\mathbb{R}}:=M\otimes_{\mathbb{Z}}\mathbb{R}\).
\item Let \(\Sigma\) be a fan in \(N_{\mathbb{R}}\) and \(X_{\Sigma}\) be
the toric variety determined by \(\Sigma\). 
Let \(T\subset X_{\Sigma}\) be its maximal torus 
with coordinates \(t_{1},\ldots,t_{n}\).
\item We denote by \(\Sigma(k)\) the set of \(k\)-dimensional cones in \(\Sigma\).
In particular, \(\Sigma(1)\) is the set of \(1\)-cones in \(\Sigma\). 
Similarly, for a cone \(\sigma\in\Sigma\),
we denote by \(\sigma(1)\) the set of \(1\)-cones belonging to \(\sigma\).
By abuse of the notation, we also denote by \(\rho\) 
the primitive generator of the corresponding 
\(1\)-cone.
\item Each \(\rho\) determines a \(T\)-invariant Weil divisor on \(X_{\Sigma}\),
which is denoted by \(D_{\rho}\) hereafter.
Any \(T\)-invariant Weil divisor \(D\) is of the form
\(D=\sum_{\rho\in\Sigma(1)} a_{\rho}D_{\rho}\). The polyhedron of \(D\) is defined to be
\begin{equation*}
\Delta_{D}:=\left\{m\in M_{\mathbb{R}}\colon \langle m,\rho\rangle\ge -a_{\rho}~
\mbox{for all}~\rho\right\}.
\end{equation*}
The integral points \(M\cap\Delta_{D}\) gives rise to a canonical
basis of \(\mathrm{H}^{0}(X_{\Sigma},D)\).
\item A \emph{nef-partition} on \(X_{\Sigma}\) 
is a decomposition of \(\Sigma(1)=\sqcup_{k=1}^{r} I_{k}\)
such that \(E_{k}:=\sum_{\rho\in I_{k}} D_{\rho}\) is nef for each \(k\). Recall that 
a divisor \(D\) is called nef if \(D.C\ge 0\) for any irreducible
complete curve \(C\subset X_{\Sigma}\).
We also have \(E_{1}+\cdots+E_{r}=-K_{X_{\Sigma}}\).
\item A polytope in \(M_{\mathbb{R}}\) is called a \emph{lattice polytope}
if its vertices belong to \(M\). For a lattice polytope \(\Delta\)
in \(M_{\mathbb{R}}\), we denote by \(\Sigma_{\Delta}\) the normal fan of 
\(\Delta\). The toric variety determined by \(\Delta\) is denoted by \(\mathbf{P}_{\Delta}\),
i.e., \(\mathbf{P}_{\Delta}=X_{\Sigma_{\Delta}}\).
\item A \emph{reflexive polytope} \(\Delta\subset M_{\mathbb{R}}\) is a lattice polytope 
containing the origin \(\mathbf{0}\in M_{\mathbb{R}}\) in its 
interior and such that the polar dual 
\(\Delta^{\vee}\) is again a lattice polytope. If
\(\Delta\) is a reflexive polytope, then \(\Delta^{\vee}\) is also a lattice
polytope and satisfies \((\Delta^{\vee})^{\vee}=\Delta\). The normal fan of \(\Delta\)
is the face fan of \(\Delta^{\vee}\) and vice versa.
\end{itemize}

\subsection{The Batyrev--Borisov duality construction}
\label{subsection:b-b-construction}
We briefly recall the construction of the dual nef-partition 
\cite{1996-Batyrev-Borisov-on-calabi-yau-complete-intersections-in-toric-varieties}.
Let \(I_{1},\ldots,I_{r}\) be a nef-partition on \(\mathbf{P}_{\Delta}\).
This gives rise to a Minkowski sum decomposition
\(\Delta=\Delta_{1}+\cdots+\Delta_{r}\), where \(\Delta_{i}=\Delta_{E_{i}}\)
is the section polytope of \(E_{i}\).
Following Batyrev--Borisov, let \(\nabla_{k}\) be the 
convex hull of \(\{\mathbf{0}\}\cup I_{k}\) and
\(\nabla=\nabla_1+\cdots+\nabla_{r}\) be their Minkowski sum.
One can prove that \(\nabla\) is a reflexive polytope in \(N_{\mathbb{R}}\)
whose polar dual is \(\nabla^{\vee}=\mathrm{Conv}(\Delta_{1},\ldots,\Delta_{r})\)
and \(\nabla_1+\cdots+\nabla_{r}\) corresponds to a nef-partition on \(\mathbf{P}_{\nabla}\),
called the \emph{dual nef-partition}.
The corresponding nef toric divisors are denoted by \(F_{1},\ldots,F_{r}\).
Then the section polytope of \(F_{j}\) is \(\nabla_{j}\).

Let \(X\to \mathbf{P}_{\Delta}\) and \(X^{\vee}\to \mathbf{P}_{\nabla}\)
be maximal projective crepant partial (MPCP for short hereafter) resolutions
for \(\mathbf{P}_{\Delta}\) and \(\mathbf{P}_{\nabla}\).
Via pullback, the nef-partitions on \(\mathbf{P}_{\Delta}\) and \(\mathbf{P}_{\nabla}\)
determine nef-partitions on \(X\) and \(X^{\vee}\) and they determine 
the families of Calabi--Yau complete intersections in \(X\) and \(X^{\vee}\) respectively.

Recall that the section polytopes \(\Delta_{i}\) and \(\nabla_{j}\)
correspond to \(E_{i}\) on \(\mathbf{P}_{\Delta}\) and 
\(F_{j}\) on \(\mathbf{P}_{\nabla}\), respectively.
To save the notation, the corresponding nef-partitions and toric divisors 
on \(X\) and \(X^{\vee}\) will be still denoted by \(\Delta_{i}\), \(\nabla_{j}\) and
\(E_{i}\), \(F_{j}\) respectively.

\subsection{Calabi--Yau double covers} 
\label{subsection:cy-double-covers}
We briefly review the construction of Calabi--Yau double covers in
\cite{2020-Hosono-Lee-Lian-Yau-mirror-symmetry-for-double-cover-calabi-yau-varieties}.
Let \(\Delta=\Delta_{1}+\cdots+\Delta_{r}\) and \(\nabla=\nabla_{1}+\cdots+\nabla_{r}\)
be a dual pair of nef-partitions  
representing \(E_{1}+\cdots+E_{r}\) on \(-K_{\mathbf{P}_{\Delta}}\)
and \(F_{1}+\cdots+F_{r}\) on \(-K_{\mathbf{P}_{\nabla}}\) respectively.
Let \(X\) and \(X^{\vee}\) be the MPCP resolution of \(\mathbf{P}_{\Delta}\) and \(\mathbf{P}_{\nabla}\) respectively.
Hereafter, we will simply call the decomposition \(\Delta=\Delta_{1}+\cdots+\Delta_{r}\)
a nef-partition on \(X\) for short with understanding the nef-partition \(E_{1}+\cdots +E_{r}\)
and likewise for the decomposition \(\nabla=\nabla_{1}+\cdots+\nabla_{r}\).
Unless otherwise stated, we assume that 

\begin{center}
{\it \(X\) and \(X^{\vee}\) are both smooth}.
\end{center}
Equivalently, we assume that both \(\Delta\) and \(\nabla\) admit uni-modular triangulations.
From the duality, we have
\begin{equation*}
\mathrm{H}^{0}(X^{\vee},F_{i})\simeq \bigoplus_{\rho\in\nabla_{i}\cap N}\mathbb{C}\cdot t^{\rho}~\mbox{and}~
\mathrm{H}^{0}(X,E_{i})\simeq \bigoplus_{m\in\Delta_{i}\cap M}\mathbb{C}\cdot t^{m}.
\end{equation*}
Here we use the same notation \(t=(t_{1},\ldots,t_{n})\) to 
denote the coordinates on the maximal torus of \(X^{\vee}\) and \(X\).

A double cover \(Y^{\vee}\to X^{\vee}\) has trivial 
canonical bundle if and only if 
the branch locus is linearly equivalent to \(-2K_{X^{\vee}}\).
Let \(Y^{\vee}\to X^{\vee}\) be 
the double cover
constructed from the section \(s=s_{1}\cdots s_{r}\) with
\begin{equation*}
(s_{1},\ldots,s_{r})\in \mathrm{H}^{0}(X^{\vee},2F_{1})\times\cdots\times
\mathrm{H}^{0}(X^{\vee},2F_{r}).
\end{equation*}

We assume that \(s_{i}\in\mathrm{H}^{0}(X^{\vee},2F_{i})\) is 
of the form \(s_{i}=s_{i,1}s_{i,2}\) with \(s_{i,1},s_{i,2}\in \mathrm{H}^{0}(X^{\vee},F_{i})\).
We further assume that \(s_{i,1}\) is the section corresponding to the lattice point
\(\mathbf{0}\in\nabla_{i}\cap N\), i.e., the scheme-theoretic zero of \(s_{i,1}\)
is \(F_{i}\), and that the scheme-theoretic zero of \(s_{i,2}\) is non-singular. 
Deforming \(s_{i,2}\), we obtain a subfamily of double covers of \(X^{\vee}\) branched over
the nef-partition  parameterized by an open subset
\begin{equation*}
V\subset \mathrm{H}^{0}(X^{\vee},F_{1})
\times\cdots\times\mathrm{H}^{0}(X^{\vee},F_{r}).
\end{equation*}

\begin{definition}
\label{definition:gauged-fixed-double-cover-family}
Given a decomposition \(\nabla=\nabla_{1}+\cdots+\nabla_{r}\) representing
a nef-partition \(F_{1}+\cdots+F_{r}\) on \(X^{\vee}\),
the subfamily \(\mathcal{Y}^{\vee}\to V\) 
constructed above is called the \emph{family of gauge fixed double covers of \(X^{\vee}\) branched over
the nef-partition} or simply the \emph{gauge fixed double cover family} if 
no confuse occurs.
\end{definition}

Given a decomposition \(\nabla=\nabla_{1}+\cdots+\nabla_{r}\) representing
a nef-partition \(F_{1}+\cdots+F_{r}\) on \(X^{\vee}\) as above,
we denote by \(\mathcal{Y}^{\vee}\to V\) the gauge fixed double cover family. 
A parallel construction is applied for the dual 
decomposition \(\Delta=\Delta_{1}+\cdots+\Delta_{r}\) representing 
the dual nef-partition \(E_{1}+\cdots+E_{r}\) 
over \(X\) and this yields
another family \(\mathcal{Y}\to U\),
where \(U\) is an open subset in 
\begin{equation*}
\mathrm{H}^{0}(X,E_{1})\times\cdots\times\mathrm{H}^{0}(X,E_{r}).
\end{equation*}

\subsection{Notation and conventions}
\label{subsection:notation}
Let us fix the notation and conventions we are going to use throughout this note.
We resume the situation and notation in \S\ref{subsection:b-b-construction}.
\begin{itemize}
\item Let \(X\to\mathbf{P}_{\Delta}\) be a MPCP resolution and 
\(\Sigma\) be the fan defining \(X\). We will assume throughout this note that
both \(X\) and \(X^{\vee}\) are \emph{smooth}.
\item Let \(I_{1},\ldots,I_{r}\) be the induced nef-partition on \(X\) as before. 
We label the elements in \(I_{k}\) by \(i_{k,1},\ldots,i_{k,n_k}\) where 
\(n_{k}=\#I_{k}\). We define \(p=n_{1}+\cdots+n_{r}\).
We will write
\begin{equation*}
\Sigma(1)=\left\{\rho_{i,j}\right\}_{1\le i\le r,~1\le j\le n_{i}}.
\end{equation*}
For convenience, we will also write \(D_{i,j}\) for the 
Weil divisor associated with \(\rho_{i,j}\).
\item Let \(\nu_{i,j}:=(\rho_{i,j},\delta_{1,i},\ldots,\delta_{r,i})
\in N\times\mathbb{Z}^{r}\) be the lifting of \(\rho_{i,j}\),
where \(\delta_{i,j}\) is the Kronecker delta.
We additionally put 
\(\nu_{i,0}:=(\mathbf{0},\delta_{1,i},\ldots,\delta_{r,i})
\in N\times\mathbb{Z}^{r}\) for \(1\le i\le r\).
\item We define an order on the set of double indexes 
by declaring \((i,j)\preceq (i',j')\)
if and only if \(i\le i'\) or \(i=i'\) and \(j\le j'\).
Recall that \(\#\{(i,j)\colon 1\le i\le r,~0\le j\le n_{i}\}=p+r\).
There are unique bijections
\begin{align*}
\begin{split}
 J:=\{(i,j)\colon 1\le i\le r,~0\le j\le n_{i}\} &\to \{1,\ldots,p+r\}
\subset (\mathbb{Z},\le ),\\
 I:=\{(i,j)\colon 1\le i\le r,~1\le j\le n_{i}\} &\to \{1,\ldots,p\}
\subset (\mathbb{Z},\le ),
\end{split}
\end{align*}
preserving the order.
\item For a positive integer \(s\) and a matrix 
\(A_{\mathrm{ext}}\in\mathrm{Mat}_{s\times (p+r)}(\mathbb{Z})\) (resp.~ 
\(A\in\mathrm{Mat}_{s\times p}(\mathbb{Z})\)),
we will label the columns of \(A_{\mathrm{ext}}\) by \(J\) 
(resp.~the columns of \(A\) by \(I\))
and speak the \((k,l)\)\textsuperscript{th} column of \(A_{\mathrm{ext}}\)
instead of the \((\sum_{1\le i\le k-1}(n_{i}+1)+l+1)\)\textsuperscript{th} 
column of \(A_{\mathrm{ext}}\) 
(resp.~the \((k,l)\)\textsuperscript{th} column of \(A\) instead of the
\((\sum_{1\le i\le k-1}n_{i}+l)\)\textsuperscript{th} column of \(A\)). 
For instance, 
for \(A_{\mathrm{ext}}\in\mathrm{Mat}_{s\times (p+r)}(\mathbb{Z})\), 
the \((1,0)\)\textsuperscript{th} column of \(A_{\mathrm{ext}}\) is the 
\(1\)\textsuperscript{st} column of \(A_{\mathrm{ext}}\).
The \((r,n_{r})\)\textsuperscript{th} column of \(A_{\mathrm{ext}}\) is the 
last column of \(A_{\mathrm{ext}}\).
\item Define the matrices
\begin{align*}
&A:=
\begin{bmatrix}
\nu_{1,1}^{\intercal} & \cdots & \nu_{r,n_{r}}^{\intercal}
\end{bmatrix}\in\mathrm{Mat}_{(n+r)\times p}(\mathbb{Z}),\\
&A_{\mathrm{ext}}:=
\begin{bmatrix}
\nu_{1,0}^{\intercal} & \cdots & \nu_{r,n_{r}}^{\intercal}
\end{bmatrix}\in\mathrm{Mat}_{(n+r)\times(p+r)}(\mathbb{Z}).
\end{align*}
According to our convention, the columns of \(A\) are labeled by \( I\)
and the columns of \(A_{\mathrm{ext}}\) are labeled by \( J\).
We have the following commutative diagram
\begin{equation*}
\begin{tikzcd}
& \mathbb{Z}^{p+r}\ar[d]\ar[r,"A_{\mathrm{ext}}"] & \mathbb{Z}^{n+r}\ar[d]\\
& \mathbb{Z}^{p}\ar[r,"A"] & \mathbb{Z}^{n}.
\end{tikzcd}
\end{equation*}
The left vertical map is given by forgetting the \((i,0)\)\textsuperscript{th}
component for all \(1\le i\le r\). The right vertical map
is given by projecting to the first \(n\) coordinates.
By assumption, \(A_{\mathrm{ext}}\) and \(A\) are surjective.
Let \(L_{\mathrm{ext}}:=\mathrm{ker}(A_{\mathrm{ext}})\) and \(L=\mathrm{ker}(A)\).
We then have
\begin{equation*}
\begin{tikzcd}
&0\ar[r] &L_{\mathrm{ext}}\ar[r]\ar[d] 
& \mathbb{Z}^{p+r}\ar[d]\ar[r,"A_{\mathrm{ext}}"] & \mathbb{Z}^{n+r}\ar[d]\ar[r] &0\\
&0\ar[r] &L\ar[r]
& \mathbb{Z}^{p}\ar[r,"A"] & \mathbb{Z}^{n}\ar[r] &0
\end{tikzcd}
\end{equation*}
where the leftmost vertical arrow is an isomorphism.

\item Each element \(\ell\in \mathbb{Z}^{s}\) can be 
uniquely written as \(\ell^{+}-\ell^{-}\) where
\(\ell^{\pm}\in\mathbb{Z}^{s}_{\ge 0}\) whose supports are disjoint.
\end{itemize}

\subsection{GKZ \texorpdfstring{\(A\)}{}-hypergeometric systems}
\label{subsection:gkz}

We adapt the notation in \S\ref{subsection:notation}. 
For \(1\le i\le r\), let \(W_i = \mathbb{C}^{n_i+1}\). 
Let \(x_{i,0},\ldots, x_{i,n_i}\) be a fixed coordinate system on the
dual space \({W_i}^{\vee}\). Set
\(\partial_{i,j}=\partial/\partial x_{i,j}\).
Given the matrix \(A_{\mathrm{ext}}\) as above
and a parameter \(\beta \in \mathbb{C}^{n+r}\), 
the \(A\)-hypergeometric ideal \(I({A_{\mathrm{ext}}},\beta)\)
is the left ideal of the Weyl algebra \(\mathscr{D}=\mathbb{C}[x,\partial]\) on
the \emph{dual} vector space \(W^{\vee}:=W_{1}^{\vee}\times\cdots\times
W_{r}^{\vee}\) generated by the following two types of
operators
\begin{itemize}
\itemsep=3pt
\item The ``box operators'': \(\partial^{\ell^{+}} - \partial^{\ell^{-}}\),
where \(\ell^{\pm}\in \mathbb{Z}_{\geq 0}^{p+r}\) satisfy 
\(A_{\mathrm{ext}}\ell^{+}=A_{\mathrm{ext}}\ell^{-}\).
Here the multi-index convention is used.
\item The ``Euler operators'': \(\mathscr{E}_{k} - \beta_{k}\), where
\(\mathscr{E}_k=\sum_{(i,j)\in J}\langle \nu_{i,j},\mathrm{e}_k\rangle 
x_{i,j}\partial_{i,j}\).
Here \(\mathrm{e}_k=(\delta_{k,1},\ldots,\delta_{k,n+r})\in\mathbb{Z}^{n+r}\).
\end{itemize}
The \(A\)-hypergeometric system \(\mathcal{M}(A_{\mathrm{ext}},\beta)\) 
is the cyclic \(\mathscr{D}\)-module
\(\mathscr{D}\slash I({A_{\mathrm{ext}}},\beta)\).
As shown by
Gel'fand~et~al.~\cite{1989-Gelfand-Zelevinskii-Kapranov-hypergeometric-functions-and-toric-varieties},
\(\mathcal{M}(A_{\mathrm{ext}},\beta)\) is a holonomic \(\mathscr{D}\)-module.

\subsection{Affine period integrals}
\label{subsection:affine-period-integrals}
Let \(\mathcal{Y}^{\vee}\to V\) be the gauge fixed double cover family
constructed in \S\ref{subsection:cy-double-covers}.
Fix a reference fiber \(Y^{\vee}=\mathcal{Y}^{\vee}_{\bullet}\) and let 
\(R\) be the branch locus of the cover 
\(\pi\colon Y^{\vee}\to X^{\vee}\).
Instead of looking at the integral of
the holomorphic top form on \(Y^{\vee}\)
over cycles in \(\mathrm{H}_{n}(Y^{\vee},\mathbb{C})\),
we will work over the maximal torus
and consider affine period integrals.
%The restriction of \(\pi\) to \(Y^{\vee}\setminus R\),
%which is again denoted by \(\pi\), is an \'{e}tale double cover
%of \(X^{\vee}\setminus R\).
%The local system \(\pi_{\ast}\mathbb{C}_{Y^{\vee}\setminus R}\)
%is decomposed into eigensheaves
%\begin{equation*}
%\pi_{\ast}\mathbb{C}_{Y^{\vee}\setminus R} = 
%\mathscr{L}_{\chi_{0}}\oplus\mathscr{L}_{\chi_{1}}
%\end{equation*}
%according to the \(\boldsymbol{\mu}_{2}\)-action. Here 
%\(\chi_{k}(a) = a^{k}\) is the character and \(a\) is the multiplicative generator of 
%\(\boldsymbol{\mu}_{2}\).

\begin{definition}
For a gauge fixed double cover family \(\mathcal{Y}^{\vee}\to V\) as above, we define
\emph{affine period integrals} to be  
\begin{equation}
\label{eq:general-period}
\Pi_{\gamma}(x):=\int_{\gamma} 
\frac{1}{s_{1,2}^{1/2}\cdots s_{r,2}^{1/2}}\frac{\mathrm{d}t_{1}}{t_{1}}\wedge\cdots
\wedge\frac{\mathrm{d}t_{n}}{t_{n}},
\end{equation}
where \(\gamma\in\mathrm{H}_{n}(X^{\vee}\setminus R,\mathscr{E})\)
and \(s_{i,2} = x_{i,0}+\sum_{j=1}^{n_{i}} x_{i,j}t^{\rho_{i,j}}\in 
\mathrm{H}^{0}(X^{\vee},F_{i})=W_{i}^{\vee}\) 
is the universal section. 
Here  
\(\mathscr{E}\) is the local system 
over \(X^{\vee}\setminus R=(\mathbb{C}^{\ast})^{n}\setminus
\cup_{i=1}^{r}\{s_{i,2}=0\}\)
whose
monodromy exponent around \(\{s_{i,2}=0\}\)
is \(1/2\).
We also define the \emph{normalized affine period integrals} to be 
\(\bar{\Pi}_{\gamma}(x):=(\textstyle\prod_{i=1}^{r} x_{i,0})^{1/2}
\Pi_{\gamma}(x)\).
\end{definition}
Note that the integrand is also multi-valued. 
The precise meaning of the integral 
\eqref{eq:general-period} is explained in 
{\cite{1990-Gelfand-Kapranov-Zelevinsky-generalized-euler-integrals-and-a-hypergeometric-functions}*{\S2.2}}.
Set \(\mathcal{A}_{\mathrm{ext}}=\{\nu_{i,j}\colon (i,j)\in J\}\).
We identify \(\mathbb{C}^{\mathcal{A}_{\mathrm{ext}}}\) with 
\begin{equation*}
W^{\vee}:=W_{1}^{\vee}\times\cdots\times W_{r}^{\vee}=\mathrm{H}^{0}(X^{\vee},F_{1})
\times\cdots\times\mathrm{H}^{0}(X^{\vee},F_{r}).
\end{equation*}
Then the affine period integrals \eqref{eq:general-period} form
a local system on \(\mathbb{C}^{\mathcal{A}_{\mathrm{ext}}}\setminus\mathcal{D}\) for 
some closed subset \(\mathcal{D}\) and in general have monodromies.

From the explicit form in \eqref{eq:general-period}, it is straightforward to see that
\begin{proposition}
\label{prop:fractional-complete-intersection}
The affine period integrals satisfy the GKZ system 
\(\mathcal{M}(A_{\mathrm{ext}},\beta)\) with 
\begin{equation*}
\beta=
\begin{bmatrix}
\mathbf{0} & -1/2 & \cdots & -1/2
\end{bmatrix}^{\intercal}\in\mathbb{C}^{n+r}.
\end{equation*}
\end{proposition}
In the region \(\mathcal{R}:=\left\{x\in\mathbb{C}^{\mathcal{A}_{\mathrm{ext}}}\colon 
|x_{i,0}|\gg \max_{j}\{|x_{i,j}|\right\}
~\mbox{for all}~i=1,\ldots,r\}\), 
by making use of the power series expansion 
\begin{equation*}
\frac{1}{\sqrt{1-w}} = \sum_{k\ge 0} r_{k} w^{k},~\mbox{for}~|w|\ll 1
\end{equation*}
we can write
\begin{align*}
\left(\frac{x_{i,0}}{s_{i,2}}\right)^{1/2}=
\sum_{k\ge 0} \frac{r_{k}}{x_{i,0}^{k}}
(-x_{i,1}t^{\rho_{i,1}}-\cdots-x_{i,n_{i}}t^{\rho_{i,n_{i}}})^{k}.
\end{align*}
The \emph{normalized affine period integrals} 
\(\bar{\Pi}_{\gamma}(x)\) become
\begin{align}
\label{eq:normalized-periods}
\bar{\Pi}_{\gamma}(x)=\int_{\gamma} \left(\prod_{i=1}^{r} \sum_{k\ge 0} 
\frac{r_{k}}{x_{i,0}^{k}}
(-x_{i,1}t^{\rho_{i,1}}-\cdots-x_{i,n_{i}}t^{\rho_{i,n_{i}}})^{k}\right)
\frac{\mathrm{d}t_{1}}{t_{1}}\wedge\cdots
\wedge\frac{\mathrm{d}t_{n}}{t_{n}}.
\end{align}
%Notice that in the region \(\mathcal{R}\), 
%the factor \((\textstyle\prod_{i=1}^{r} x_{i,0})^{1/2}\)
%can be treated as a trivialization of the local system \(\mathscr{L}_{\chi_{1}}\)
%as well so \(\gamma\) in \eqref{eq:normalized-periods} can
%be regarded as a cycle with \emph{complex} coefficients.

Consider the cycle \(\gamma_{0}:=\left\{|t_{1}|=\cdots=|t_{n}|=\epsilon\right\}\).
We can compute \(\bar{\Pi}_{\gamma_{0}}(x)\).
Using the residue formula, over the region \(\mathcal{R}\), we have
\begin{align}
\label{eq:period-power-expansion}
\bar{\Pi}_{\gamma_{0}}(x)=(2\pi \sqrt{-1})^{n} 
\sum_{\ell\in\mathfrak{L}}
C_{\ell}\cdot (-1)^{\sum_{i=1}^{r}\ell_{i,0}}x^{\ell}
\end{align}
where \(\mathfrak{L}:=\{\ell\in L_{\mathrm{ext}}
\colon\ell_{i,j}\ge 0~\mbox{for all}~j\ne 0\}\)
and
\begin{align}
\label{eq:def-cl}
C_{\ell}= \prod_{i=1}^{r}
\frac{r_{-\ell_{i,0}}\Gamma(-\ell_{i,0}+1)}{\Gamma(\ell_{i,1}+1)\cdots
\Gamma(\ell_{i,n_{i}}+1)}.
\end{align}
\begin{remark}
\label{remark:period-cycle}
The sheaf \(\pi_{\ast}\mathbb{C}_{Y^{\vee}}\) 
(resp.~\(\pi_{\ast}\mathbb{C}_{Y^{\vee}\setminus R}\)) is decomposed into
eigensheaves
\begin{equation*}
\pi_{\ast}\mathbb{C}_{Y^{\vee}} = \mathscr{G}_{\chi_{0}}\oplus\mathscr{G}_{\chi_{1}}
~(\mbox{resp}.~\pi_{\ast}\mathbb{C}_{Y^{\vee}\setminus R} = 
\mathscr{L}_{\chi_{0}}\oplus\mathscr{L}_{\chi_{1}}).
\end{equation*} 
Here \(\chi_{k}(a)=a^{k}\) where \(a\)
is the generator of the multiplicative group \(\mathbb{Z}\slash 2\mathbb{Z}\).
Let \(i\colon R\to Y^{\vee}\) and 
\(j\colon Y^{\vee}\setminus R\to Y^{\vee}\) be the closed and open embedding.
Consider the standard triangle in the derived category
\begin{equation*}
j_{!}\mathbb{C}_{Y^{\vee}\setminus R}\to \mathbb{C}_{Y^{\vee}}
\to i_{+}\mathbb{C}_{R}.
\end{equation*}
Applying the functor \(R\pi_{\ast}\) to the above sequence, 
one can show that \(\left. \mathscr{G}_{\chi_{1}}\right|_{X^{\vee}\setminus R}
\simeq \mathscr{L}_{\chi_{1}}\) and that 
\(\mathrm{H}^{n}_{\mathrm{c}}(X^{\vee},\mathscr{G}_{\chi_{1}})\simeq
\mathrm{H}_{\mathrm{c}}^{n}(X^{\vee}\setminus R,\mathscr{L}_{\chi_{1}})\).

Moreover, we have 
\begin{align*}
\mathrm{H}^{n}_{\mathrm{c}}(Y^{\vee},\mathbb{C})&=
\mathrm{H}^{n}_{\mathrm{c}}(X^{\vee},\mathscr{G}_{\chi_{0}})\oplus
\mathrm{H}^{n}_{\mathrm{c}}(X^{\vee},\mathscr{G}_{\chi_{1}})\\
&=\mathrm{H}^{n}_{\mathrm{c}}(X^{\vee},\mathbb{C})\oplus
\mathrm{H}^{n}_{\mathrm{c}}(X^{\vee},\mathscr{G}_{\chi_{1}})\\
&\simeq\mathrm{H}^{n}_{\mathrm{c}}(X^{\vee},\mathbb{C})\oplus
\mathrm{H}_{\mathrm{c}}^{n}(X^{\vee}\setminus R,\mathscr{L}_{\chi_{1}}).
\end{align*}
If \(n\) is odd, \(\mathrm{H}^{n}_{\mathrm{c}}(X^{\vee},\mathbb{C})=0\)
since \(X^{\vee}\) is a smooth toric variety. 
\end{remark}

\section{Existence of maximal degeneracy points}
In this section, we study the 
maximal degeneracy problem and show that the GKZ
system associated with the gauge fixed double cover family
\(\mathcal{Y}^{\vee}\to V\) 
admits a maximal degeneracy point on a resolution of the secondary fan compactification of the moduli.
This extends the results in 
\cite{1997-Hosono-Lian-Yau-maximal-degeneracy-points-of-gkz-systems} to our case. 
The proof presented here is parallel to the one given in 
\cite{1997-Hosono-Lian-Yau-maximal-degeneracy-points-of-gkz-systems}.

\subsection{The maximal degeneracy points}
From the discussion in \S\ref{subsection:affine-period-integrals},
the affine period integrals \eqref{eq:general-period}
are sections of a local system defined on   
\(\mathbb{C}^{\mathcal{A}_{\mathrm{ext}}}\setminus\mathcal{D}\).
Recall that \(\mathcal{A}_{\mathrm{ext}}=\{\nu_{i,j}\colon (i,j)\in J\}\), 
\(W_{i}=\mathrm{H}^{0}(X^{\vee},F_{i})^{\vee}\), and \(W=\prod_{i=1}^{r} W_{i}\).
We also identify \(\mathbb{C}^{\mathcal{A}_{\mathrm{ext}}}\) with \(W^{\vee}\).
Applying the functor \(\mathrm{Hom}_{\mathbb{Z}}(-,\mathbb{C}^{\ast})\) to 
the sequence
\begin{equation*}
0\to L_{\mathrm{ext}}\to \mathbb{Z}^{p+r}\equiv
\mathbb{Z}^{\mathcal{A}_{\mathrm{ext}}}\to \mathbb{Z}^{n+r}\equiv N\times\mathbb{Z}^{r}\to 0,
\end{equation*}
we obtain a short exact sequence of algebraic tori 
(\(T_{M}=\mathrm{Hom}_{\mathbb{Z}}(N,\mathbb{C}^{\ast})\))
\begin{equation*}
1\to T_{M}\times(\mathbb{C}^{\ast})^{r}\to (\mathbb{C}^{\ast})^{\mathcal{A}_{\mathrm{ext}}}
\to \mathrm{Hom}(L_{\mathrm{ext}},\mathbb{C}^{\ast})\to 1.
\end{equation*}
Let \(S_{\mathcal{A}_{\mathrm{ext}}}\) be the image of 
\((\mathbb{C}^{\ast})^{\mathcal{A}_{\mathrm{ext}}}\setminus\mathcal{D}\)
under the map
\begin{equation*}
(\mathbb{C}^{\ast})^{\mathcal{A}_{\mathrm{ext}}}\to
(\mathbb{C}^{\ast})^{\mathcal{A}_{\mathrm{ext}}}\slash T_{M}\times (\mathbb{C}^{\ast})^{r}
\xrightarrow{\phi} 
\mathrm{Hom}_{\mathbb{Z}}(L_{\mathrm{ext}},\mathbb{C}^{\ast}). 
\end{equation*}
Here the isomorphism \(\phi\) is given by 
\begin{equation*}
\phi(x)(\ell) = (-1)^{\sum_{i=1}^{r}\ell_{i,0}}x^{\ell}~
\mbox{where}~
\ell\in L_{\mathrm{ext}}.
\end{equation*}
Any complete fan \(F\) in \(L_{\mathrm{ext}}^{\vee}\otimes\mathbb{R}\) 
gives rise to a complete toric variety \(X_{F}\)
which compactifies the torus 
\begin{equation*}
\mathrm{Hom}_{\mathbb{Z}}(L_{\mathrm{ext}},\mathbb{C}^{\ast})
\simeq \mathrm{Hom}_{\mathbb{Z}}(L,\mathbb{C}^{\ast})
\end{equation*}
and \(S_{\mathcal{A}_{\mathrm{ext}}}\) as well. 
Since the normalized affine period integrals \(\bar{\Pi}_{\gamma}(x)\)
are \(T_{M}\times (\mathbb{C}^{\ast})^{r}\) invariant, 
they descend to local sections of a locally constant sheaf on \(S_{\mathcal{A}_{\mathrm{ext}}}\).

\begin{definition}
\label{def:maximal-degeneracy-points}
We call a smooth boundary point 
\(p\in X_{F}\setminus \mathrm{Hom}_{\mathbb{Z}}(L,\mathbb{C}^{\ast})\) a 
\emph{maximal degeneracy point} if near \(p\) there is exactly one normalized affine period integral 
\(\bar{\Pi}_{\gamma}\) (up to a constant) extends over \(p\) holomorphically.
\end{definition}

\subsection{Triangulations, secondary fans and Gr\"{o}bner fans}
To proceed, let us retain the notation in 
\S\ref{subsection:notation} and recall the following terminologies. 
\begin{itemize}
\item Let \(\mathcal{A}_{\mathrm{ext}}=\{\nu_{i,j}\colon (i,j)\in J\}\) 
be the set points in \(\mathbb{Z}^{n+r}\).
We denote by \(\mathrm{Conv}(\mathcal{A}_{\mathrm{ext}})\) the convex 
hull generated by \(\mathcal{A}_{\mathrm{ext}}\).
\item A triangulation \(\mathscr{T}\) of \(\mathrm{Conv}(\mathcal{A}_{\mathrm{ext}})\) is a 
collection of \((r+n-1)\)-dimensional simplices whose
vertices are in \(\mathcal{A}_{\mathrm{ext}}\) such that the 
intersection of two such simplices is a face of both
and that their union is \(\mathrm{Conv}(\mathcal{A}_{\mathrm{ext}})\).
\item A continuous function \(h\) on the cone over \(\mathrm{Conv}(\mathcal{A}_{\mathrm{ext}})\) is called 
\emph{\(\mathscr{T}\)-piecewise linear} if it is linear on the cone over each 
simplex in \(\mathscr{T}\).
A \(\mathscr{T}\)-piecewise linear function \(h\) is called \emph{convex} 
if \(h(a+b)\le h(a)+h(b)\)
for arbitrary \(a,b\) and is called \emph{strictly convex} 
if it is convex and
\(\left.h\right|_{\sigma}\ne \left.h\right|_{\tau}\) for any large cones \(\sigma\ne\tau\).
\item Each point \(x\in\mathbb{R}^{\mathcal{A}_{\mathrm{ext}}}\) (components are labeled by \((i,j)\in J\))
determines a \(\mathscr{T}\)-piecewise linear function, which is denoted by \(h_{x}\). 
Let \(\mathcal{C}(\mathscr{T})\) be the set of all \(x\in\mathbb{R}^{\mathcal{A}_{\mathrm{ext}}}\) 
such that \(h_{x}\) is convex and that \(h_{x}(\nu_{i,j})\le x_{i,j}\) for a non vertex
\(\nu_{i,j}\in\mathcal{A}_{\mathrm{ext}}\).
Note that \(\mathcal{C}(\mathscr{T})\) is a rational
polyhedral cone in \(\mathbb{R}^{\mathcal{A}_{\mathrm{ext}}}\) 
but not strongly convex.
\item A triangulation \(\mathscr{T}\) is called \emph{regular} if 
\(\mathcal{C}(\mathscr{T})\) contains an interior point,
i.e., there exists an \(x\in\mathbb{R}^{\mathcal{A}_{\mathrm{ext}}}\) such that
\(h_{x}\) is a strictly convex function.
\end{itemize}

\begin{definition}
\label{def:secondary-fan}
The collection of the cones \(\mathcal{C}(\mathscr{T})\)
with \(\mathscr{T}\) regular, together
with all of their faces form a generalized fan in \(\mathbb{R}^{\mathcal{A}_{\mathrm{ext}}}\).
Note that each cone in \(\mathcal{C}(\mathscr{T})\) contains 
\(M_{\mathbb{R}}\times\mathbb{R}^{r}\) as a linear subspace
via \(A_{\mathrm{ext}}^{\intercal}\colon 
M_{\mathbb{R}}\times\mathbb{R}^{r}\hookrightarrow \mathbb{R}^{\mathcal{A}_{\mathrm{ext}}}\).
We can project the generalized fan \(\mathcal{C}(\mathscr{T})\)
along the subspace and get a complete fan in \(L_{\mathrm{ext}}^{\vee}\otimes\mathbb{R}\).
The resulting fan \(S\Sigma\) is called the \emph{secondary fan} of \(\mathcal{A}_{\mathrm{ext}}\).
\end{definition}

Each \(\omega\in\mathbb{R}^{\mathcal{A}_{\mathrm{ext}}}\) determines a polyhedral 
subdivision on \(\mathrm{Conv}(\mathcal{A}_{\mathrm{ext}})\).
Let \(C=\mathrm{Cone}\{(\nu_{i,j},\omega_{i,j})\in
\mathcal{A}_{\mathrm{ext}}\times\mathbb{R}\colon
\nu_{i,j}\in\mathcal{A}_{\mathrm{ext}}\}\). Recall that 
the lower hull of \(C\) is a 
collection of facets of \(C\) whose last coordinate in the inward 
normal vector is positive. Projecting down 
the facets in the lower hull gives rises to 
a polyhedral subdivision of \(\mathrm{Conv}(\mathcal{A}_{\mathrm{ext}})\)
if \(\dim C=n+r\).
For generic \(\omega\), the subdivision \(\mathscr{T}_{\omega}\) is a triangulation.
One can show that a triangulation \(\mathscr{T}\) of \(\mathrm{Conv}(\mathcal{A}_{\mathrm{ext}})\) is regular
if and only if \(\mathscr{T} = \mathscr{T}_{\omega}\) 
for some \(\omega\in\mathbb{R}^{\mathcal{A}_{\mathrm{ext}}}\).

Consider a polynomial ring 
\(\mathbb{C}[y]:=\mathbb{C}[y_{i,j}\colon (i,j)\in J]\)
and the toric ideal
\begin{eqnarray*}
I_{\mathcal{A}_{\mathrm{ext}}}=\left\langle y^{l^{+}}-y^{l^{-}}\colon l=l^{+}-l^{-}\in L_{\mathrm{ext}}
\right\rangle.
\end{eqnarray*}
Each \(\omega\in\mathbb{R}_{\ge 0}^{\mathcal{A}_{\mathrm{ext}}}\) 
determines a weight on \(\mathbb{C}[y]\)
by defining 
\begin{eqnarray*}
\mathrm{in}_{\omega}(y^{n}):=\sum_{i,j} \omega_{i,j} n_{i,j},~\mbox{where}
~y^{n} = \prod_{i,j} y_{i,j}^{n_{i,j}}.
\end{eqnarray*} 
Let \(\mathrm{LT}_{\omega}(I_{\mathcal{A}_{\mathrm{ext}}})\) be the leading term ideal 
with respect to \(\mathrm{in}_{\omega}\). We say that 
\(\omega,~\omega'\in\mathbb{R}^{\mathcal{A}_{\mathrm{ext}}}_{\ge 0}\) are equivalent 
if \(\mathrm{LT}_{\omega}(I_{\mathcal{A}_{\mathrm{ext}}}) =
\mathrm{LT}_{\omega'}(I_{\mathcal{A}_{\mathrm{ext}}})\).
We can extend the equivalence relation to 
\(\mathbb{R}^{\mathcal{A}_{\mathrm{ext}}}\) by the homogeneity of \(I_{\mathcal{A}_{\mathrm{ext}}}\).

\begin{definition}[cf.~\cites{1991-Sturmfels-grobner-bases-of-toric-varieties,1996-Sturmfels-grobner-bases-and-convex-polytopes}]
\label{def:grobner-fan}
The equivalence classes of vectors in \(\mathbb{R}^{\mathcal{A}_{\mathrm{ext}}}\) form a fan. 
Projecting along the linear subspace \(A_{\mathrm{ext}}^{\intercal}\colon
M_{\mathbb{R}}\times\mathbb{R}^{r}
\hookrightarrow \mathbb{R}^{\mathcal{A}_{\mathrm{ext}}}\), we obtain a fan in \(L_{\mathrm{ext}}^{\vee}
\otimes{\mathbb{R}}\).
The resulting fan \(G\Sigma\) is called the \emph{Gr\"{o}bner fan} of \(\mathcal{A}_{\mathrm{ext}}\).
An interior point
in a large cone in \(G\Sigma\) is called a \emph{term order} of \(I_{\mathcal{A}_{\mathrm{ext}}}\).
\end{definition}

\begin{remark}
Although the secondary fan and the Gr\"{o}bner fan (cf.~Definition \ref{def:secondary-fan}
and Definition \ref{def:grobner-fan}) depend not only on \(\Sigma\) but also on the nef-partition,
we still denote them by \(S\Sigma\) and \(G\Sigma\) respectively for simplicity.
We also remark that \(\mathrm{Conv}(\mathcal{A}_{\mathrm{ext}})\) projects
to \(\mathrm{Conv}(\nabla_{1},\ldots,\nabla_{r})\) under the canonical projection
\(N_{\mathbb{R}}\times\mathbb{R}^{r}\to N_{\mathbb{R}}\).
\end{remark}

\begin{remark}
Sturmfels \cite{1991-Sturmfels-grobner-bases-of-toric-varieties} showed that
the Gr\"{o}bner fan \(G\Sigma\) 
refines the secondary fan \(S\Sigma\). 
The two fans coincide if 
%{\color{red}\cancel{and only if}} 
\(\mathcal{A}_{\mathrm{ext}}\) is 
%{\color{red}\cancel{totally}} 
unimodular.
In particular, if \(\omega\in\mathbb{R}^{\mathcal{A}_{\mathrm{ext}}}\) is a term order, 
then \(\mathscr{T}_{\omega}\) is a triangulation of 
\(\mathrm{Conv}(\mathcal{A}_{\mathrm{ext}})\).
\end{remark}
\subsection{The cohomology ring of toric manifolds}
We resume the notation in \S\ref{subsection:notation}
and the situation there.
Recall that a primitive collection of \( \Sigma \) 
is a subset \( \mathcal{P}\subset \Sigma(1) \)
such that the full set \(\mathcal{P}\) does not form a cone in \(\Sigma\) but any proper
subset does. 

For a projective \emph{smooth} toric variety \(X_{\Sigma}\),
the cohomology ring \(\mathrm{H}^{\bullet}(X_{\Sigma},\mathbb{Z})\) is given by
\( \mathbb{Z}[a_{i,j}\colon (i,j)\in I]\slash \mathcal{I}\), where \(\mathcal{I}\) is 
the ideal generated by
\begin{itemize}
\itemsep=3pt
\item[(a)] \(a_{\mathcal{P}}:=\prod_{(i,j)\in\mathcal{P}} a_{i,j}\), 
where \(\mathcal{P}\) is a primitive collection
in \(\Sigma\);
\item[(b)] \(\sum_{(i,j)\in I}\langle m,\rho_{i,j}\rangle a_{i,j}\) 
for all \(m\in M\).
\end{itemize}
The ideal generated by (a) 
is called the \emph{Stanley--Reisner ideal} of \( \Sigma \).

For a primitive collection \(\mathcal{P}\), we can define the
\emph{primitive relation} of \(\mathcal{P}\) as follows.
By completeness of \(\Sigma\), 
the vector \(\sum_{(i,j)\in\mathcal{P}} \rho_{i,j}\) must lie in the relative interior of 
some cone \(\sigma\) uniquely in \(\Sigma\). We may write
\begin{equation*}
\sum_{(i,j)\in\mathcal{P}} \rho_{i,j} 
= \sum_{(i,j)\in \sigma(1)} c_{i,j} \rho_{i,j},~c_{i,j}\in\mathbb{Z}_{>0}.
\end{equation*}
Equivalently, we have
\begin{equation*}
\sum_{(i,j)\in\mathcal{P}} \rho_{i,j} 
-\sum_{(i,j)\in \sigma(1)} c_{i,j} \rho_{i,j}
=\sum_{(i,j)\in I} b_{i,j} \rho_{i,j} = 0,~\mbox{with}~b_{i,j}\in\mathbb{Z}.
\end{equation*}
Under the inclusion \(L\hookrightarrow \mathbb{R}^{p}\),
the vector \( (b_{i,j})\in \mathbb{R}^{p} \) is an 
element in \(L\), called the \emph{primitive relation of \(\mathcal{P}\)},
and is denoted by \(\ell(\mathcal{P})\). We can identify 
\(L\otimes\mathbb{R}\) with \(\mathrm{N}_{1}(X_{\Sigma})\), 
the real vector space of \(1\)-cycles on \(X_{\Sigma}\)
modulo numerical equivalence.
\begin{proposition}[Toric cone theorem]
Let \(\mathrm{NE}(X_{\Sigma})\subset L\otimes\mathbb{R}\) be the cone generated by
classes of irreducible complete curves in \(X_{\Sigma}\).
We have
\begin{equation}
\overline{\mathrm{NE}}(X_{\Sigma}) = \mathrm{NE}(X_{\Sigma}) = 
\sum_{\mathcal{P}} \mathbb{R}_{\ge 0} \ell(\mathcal{P}),
\end{equation}
where the summation runs over all primitive collections $\mathcal{P}$.
\end{proposition}
\begin{lemma}
Under our smoothness assumption, we have \( \mathcal{P} \cap \sigma(1) = \emptyset \)
where \(\sigma(1)\) is the set of \(1\)-cones contained in \(\sigma\). 
\end{lemma}
\begin{proof}
See {\cite{1997-Hosono-Lian-Yau-maximal-degeneracy-points-of-gkz-systems}*{Proposition 4.7}}.
\end{proof}
We can lift the primitive relations to obtain
relations among \(\nu_{i,j}\).
For a primitive collection \(\mathcal{P}\), we have correspondingly a cone \(\sigma\)
in \(\Sigma\) as above. We can thus write
\begin{equation}
\label{equation:primitive-relation-lifting}
\sum_{(i,j)\in\mathcal{P}} \nu_{i,j} = \sum_{(i,j)\in\sigma(1)} c_{i,j}\nu_{i,j} + 
\sum_{i=1}^{r} c_{i,0} \nu_{i,0}.
\end{equation}
\begin{corollary}
\(c_{i,0}\ge 0\) for all \(i=1,\ldots,r\).
\end{corollary}
\begin{proof}
\(\ell(\mathcal{P})\) represents a curve class. The assertion 
follows from the fact that \( I_{1}\sqcup\cdots\sqcup I_{r} \) is a nef-partition 
and \(c_{i,0}\) is the intersection number of \(\ell(\mathcal{P})\) with \(E_{i}\).
\end{proof}

Let \(\ell(\mathcal{P})\) be a primitive relation and \(\ell_{\mathrm{ext}}(\mathcal{P})\)
be the corresponding element in \(L_{\mathrm{ext}}\) under the identification \(L\simeq
L_{\mathrm{ext}}\). We can rewrite \eqref{equation:primitive-relation-lifting}
into
\begin{equation}
0=\sum_{(i,j)\in\mathcal{P}} \nu_{i,j} - \sum_{(i,j)\in\sigma(1)} c_{i,j}\nu_{i,j} -
\sum_{i=1}^{r} c_{i,0} \nu_{i,0}=\sum_{(i,j)\in J} d_{i,j}\nu_{i,j}.
\end{equation}
\begin{corollary}
\label{corollary:coefficient-vector-lifting-relation}
The vector \((d_{i,j})_{(i,j)\in J}\)
is equal to \(\ell_{\mathrm{ext}}(\mathcal{P})\) as
elements in \(\mathbb{R}^{\mathcal{A}_{\mathrm{ext}}}\) and 
\(\ell_{\mathrm{ext}}^{\pm}(\mathcal{P})\) is given by the left-hand
and the right-hand side of \eqref{equation:primitive-relation-lifting}.
\end{corollary}

\subsection{Indicial ideals of Picard--Fuchs equations}
Our aim in this paragraph is to
describe the indicial rings attached to the GKZ system.
The arguments here are almost along the same line in 
\cite{1997-Hosono-Lian-Yau-maximal-degeneracy-points-of-gkz-systems}.
In this subsection, unless otherwise stated, \(X=X_{\Sigma}\) is a smooth projective toric
variety defined by a fan \(\Sigma\) as in \S\ref{subsection:notation}.
\begin{definition}
For \(\ell\in L_{\mathrm{ext}}=
\mathrm{ker}(A_{\mathrm{ext}})\), we define
\begin{equation*}
I_{\ell}(\alpha):= 
x^{-\alpha} x^{\ell^{+}}
(\partial_{x})^{\ell^{+}}
x^{\alpha}\in \mathbb{C}[\alpha]:=\mathbb{C}[\alpha_{i,j}\colon (i,j)\in J].
\end{equation*}
\end{definition}
Let us recall the definition of indicial ideals.
\begin{definition}
For a cone \( \tau\subset L_{\mathrm{ext}}^{\vee}\otimes \mathbb{R} \)
and an exponent \( \beta\in\mathbb{C}^{n+r} \), 
the \emph{indicial ideal} \( \mathrm{Ind}(\tau,\beta) \)
is the ideal in \(\mathbb{C}[\alpha_{i,j}\colon (i,j)\in J]\) 
generated by
\begin{itemize}
\item \(I_{\ell}(\alpha)\) where \(
0\ne \ell\in\tau^{\vee}\cap 
L_{\mathrm{ext}}\);
\item \(\textstyle\sum_{(i,j)\in J} \langle 
\bar{m},\nu_{i,j}\rangle\alpha_{i,j} -
\langle \bar{m},\beta\rangle\) for all \(
\bar{m}\in M\times\mathbb{Z}^{r}\).
\end{itemize}
\end{definition}

There is a canonical triangulation on \(\mathrm{Conv}(\mathcal{A}_{\mathrm{ext}})\). It is
given by the maximal cones in the fan defining the toric variety \(W\), 
the total space of the rank \(r\) vector bundle over \( X \) 
whose sheaf of sections is \( \oplus_{i=1}^{r} \mathcal{O}_{X}(-E_{i})\).
(Recall that \(E_{1}+\ldots+E_{r}\) is the nef-partition on \(X\).)
We call this triangulation the 
\emph{maximal triangulation} of \( \mathrm{Conv}(\mathcal{A}_{\mathrm{ext}}) \) and
is denoted by \(\mathscr{T}_{\mathrm{max}}\).

For a smooth variety \(X\), 
the K\"{a}hler cone of \(X\) is denoted by \(\mbox{K\"{a}h}(X)\). 
If \(X\) is a \emph{smooth projective} toric variety, then
\(\mbox{K\"{a}h}(X)\) is a cone sitting inside 
\(\mathrm{H}^{2}(X,\mathbb{R})\)
whose closure 
coincides with the the closure of the ample cone \(\mathrm{Amp}(X)\).
This is a large cone since \(X\) is projective.
Let \(\overline{\mbox{K\"{a}h}(X)}\) be the closure of the K\"{a}hler cone of \(X\).

Since \(E_{1}+\cdots+E_{r}\) is a nef-partition, we have
\( \ell_{\mathrm{ext}}(\mathcal{P})_{i,0}\le 0 \) for all \(i\) and 
{\cite{1997-Hosono-Lian-Yau-maximal-degeneracy-points-of-gkz-systems}*{Proposition 6.1}}
still holds. Combining with [loc.~cit., Corollary 6.2 and Corollary 6.3], we obtain
the following corollaries.

\begin{corollary}
The leading term ideal \(\mathrm{LT}_{\omega}(I_{\mathcal{A}_{\mathrm{ext}}})\) 
with respect to the term order \(\omega\) such that
\(\mathscr{T}_{\omega}=\mathscr{T}_{\mathrm{max}}\) is 
the Stanley--Reisner ideal of \( \Sigma \).
\end{corollary}

\begin{corollary}
\(\overline{\mbox{K\"{a}h}(X)}\in G\Sigma\). 
\end{corollary}

Taking such a term order \(\omega\) (one can take
any element in the ample cone to achieve this), we see that 
the leading term ideal of 
\begin{eqnarray*}
\left\{ y^{\ell_{\mathrm{ext}}^{+}(\mathcal{P})}-
y^{\ell_{\mathrm{ext}}^{-}(\mathcal{P})}\colon \mathcal{P}~
\mbox{is a primitive collection}\right\}
\end{eqnarray*}
is nothing but the Stanley--Reisner ideal.
Indeed, since \( \mathcal{P} \) is primitive and \( \omega \) is ample, \( \omega.\ell(\mathcal{P}) > 0 \).
Consequently, \( y^{\ell_{\mathrm{ext}}^{+}(\mathcal{P})} \) is the 
leading term with respect to \(\omega\).
Now use Corollary \ref{corollary:coefficient-vector-lifting-relation}.
\begin{corollary}
\label{cor:minimal-grobner-basis}
The set 
\begin{eqnarray*}
\left\{ y^{\ell_{\mathrm{ext}}^{+}(\mathcal{P})}-
y^{\ell_{\mathrm{ext}}^{-}(\mathcal{P})}\colon \mathcal{P}~
\mbox{is a primitive collection}\right\}
\end{eqnarray*}
is a minimal 
Gr\"{o}bner basis of the toric ideal \(I_{\mathcal{A}_{\mathrm{ext}}}\) 
for any term order \(\omega\) with \(\mathscr{T}_{\omega}=\mathscr{T}_{\mathrm{max}}\).
Consequently, the polynomial operators in the GKZ system \(\mathcal{M}(A_{\mathrm{ext}},\beta)\) 
are generated by
the box operators associated with \(\ell_{\mathrm{ext}}(\mathcal{P})\),
where \(\mathcal{P}\) is a primitive collection of \(\Sigma\).
\end{corollary}

\begin{remark}
From the corollaries above,
we see that one can also use the Gr\"{o}bner basis with
respect to a term order \(\omega\) with \(\mathscr{T}_{\omega}=\mathscr{T}_{\mathrm{max}}\) to
approximate the indicial ideal
\(\mathrm{Ind}(\tau,\beta)\) 
as well as 
the GKZ system in our fractional case.
\end{remark}

Let \( \ell_{\mathrm{ext}}(\mathcal{P}) \) be the lifting of \(\ell(\mathcal{P})\) under
the isomorphism \(L_{\mathrm{ext}}\simeq L\) as before.
\begin{lemma}
\label{lemma:ideal-indicial-compare}
Let \( \tau = \overline{\mbox{K\"{a}h}(X)} \).
The ideal generated by
\begin{itemize}
\item[\textrm{(a')}] \(I_{\ell_{\mathrm{ext}}(\mathcal{P})}(\alpha)\) 
for \(\mathcal{P}\) primitive;
\item[\textrm{(b')}] \(\textstyle\sum_{(i,j)\in J} \langle \bar{m},\nu_{i,j}
\rangle\alpha_{i,j} -\langle \bar{m},\beta\rangle\) for all \(
\bar{m}\in M\times\mathbb{Z}^{r}\);
\end{itemize}
is an ideal contained in \( \mathrm{Ind}(\tau,\beta) \). Moreover,
they have the same zero locus. 
\end{lemma}
\begin{proof}
Let \(\mathcal{I}'\) be the ideal generated by the elements in 
\(\mathrm{(a')}\) and \(\mathrm{(b')}\). Clearly, we have \(\mathcal{I}'\subset
\mathrm{Ind}(\tau,\beta)\). 
For any term order \(\omega\in\tau\), 
we have \(\mathscr{T}_{\omega}=
\mathscr{T}_{\mathrm{max}}\). 
Together with Corollary \ref{cor:minimal-grobner-basis}, it follows from 
{\cite{1997-Hosono-Lian-Yau-maximal-degeneracy-points-of-gkz-systems}*{Proposition 5.6}}
that the zero locus of 
\(\mathrm{Ind}(\tau,\beta)\) is the same as the one defined by \(\mathcal{I}'\).
\end{proof}

From this, we can deduce that 
\begin{proposition}
\label{prop:zero-locus}
Let \(\tau \subset \overline{\mbox{K\"{a}h}(X)}\).
There is a surjection
\begin{equation}
\mathrm{H}^{\bullet}(X,\mathbb{C}) 
\to\mathbb{C}[\alpha]\slash \mathrm{Ind}(\tau,\beta),
~D_{i,j}\mapsto \alpha_{i,j}
\end{equation}
from the cohomology ring of \(X\) to 
the indicial ring of the GKZ \(A\)-hypergoemetric system
associated with the family \(\mathcal{Y}^{\vee}\to V\).
\end{proposition}
\begin{proof}
Let \(\mathcal{I}'\) again be the ideal generated by the elements in 
\(\mathrm{(a')}\) and \(\mathrm{(b')}\) in Lemma \ref{lemma:ideal-indicial-compare}.
By Corollary \ref{corollary:coefficient-vector-lifting-relation}, 
for a primitive collection \(\mathcal{P}\), we have
\(\ell_{\mathrm{ext}}^{+}(\mathcal{P})_{i,0}=0\) for all \(i\), 
\(\ell_{\mathrm{ext}}^{+}(\mathcal{P})_{i,j}=1\) for \(\rho_{i,j}\in\mathcal{P}\),
and \(\ell_{\mathrm{ext}}^{+}(\mathcal{P})_{i,j}=0\) for \(\rho_{i,j}\notin\mathcal{P}\).
Consequently,
\begin{equation*}
I_{\ell_{\mathrm{ext}}(\mathcal{P})}(\alpha) = \alpha^{\ell_{\mathrm{ext}}^{+}(\mathcal{P})}.
\end{equation*}
When \(\mathcal{P}\) runs through all primitive collections of \(\Sigma\),
the elements \(I_{\ell_{\mathrm{ext}}(\mathcal{P})}(\alpha)\) generate exactly 
the Stanley--Reisner ideal of \(\Sigma\). From this, we see that
\begin{equation*}
\mathrm{H}^{\bullet}(X,\mathbb{C}) \simeq \mathbb{C}[\alpha_{i,j}\colon (i,j)\in J]\slash
\mathcal{I}'.
\end{equation*}
The statement follows from the fact that 
\(\mathcal{I}'\subset\mathrm{Ind}(\overline{\mbox{K\"{a}h}(X)},\beta)
\subset\mathrm{Ind}(\tau,\beta)\).
\end{proof}

In particular, this implies
\begin{corollary}
Let \(\tau\) be as in Proposition \ref{prop:zero-locus}.
The zero locus of \(\mathrm{Ind}(\tau,\beta)\)
consists of at most one point \(\alpha=(\alpha_{i,j})
\in\mathbb{C}^{p+r}\) where \(\alpha_{i,0}=-1/2\) for \(1\le i\le r\)
and \(\alpha_{i,j}=0\) for other \(i,j\).
\end{corollary}

\subsection{The existence of maximal degeneracy points}
We summarize the results we have obtained in the previous paragraphs.
Recall that the secondary fan \(S\Sigma\) is
a complete fan in \(L^{\vee}_{\mathrm{ext}}\otimes\mathbb{R}\)
and the toric variety \(X_{S\Sigma}\) gives rise to a compactification 
of the algebraic torus 
\begin{equation*}
(\mathbb{C}^{\ast})^{\mathcal{A}_{\mathrm{ext}}}\slash 
T_{M}\times(\mathbb{C}^{\ast})^{r}.
\end{equation*}
The Gr\"{o}bner fan \(G\Sigma\) 
gives a partial resolution \(X_{G\Sigma}\to X_{S\Sigma}\).

Let \(\tau\) be a \emph{regular} maximal cone in the space 
\(L^{\vee}_{\mathrm{ext}}\otimes\mathbb{R}\). 
It determines a unique integral basis 
\(\{\ell^{(1)},\ldots,\ell^{(p-n)}\}\) of 
\(L_{\mathrm{ext}}\) in \(\tau^{\vee}\cap L_{\mathrm{ext}}\), 
and hence a set of canonical coordinates \(z_{\tau}^{(1)},\ldots,z_{\tau}^{(p-n)}\)
on the smooth affine toric variety \(X_{\tau} = \mathrm{Hom}(\tau^{\vee}\cap 
L_{\mathrm{ext}},\mathbb{C})\).
Explicitly, we have
\begin{equation*}
z_{\tau}^{(k)}= (-1)^{\sum_{i=1}^{r}\ell^{(k)}_{i,0}}
x^{\ell^{(k)}},~1\le k\le p-n.
\end{equation*}
We can employ the argument in 
{\cite{1997-Hosono-Lian-Yau-maximal-degeneracy-points-of-gkz-systems}*{Corollary 5.12}}
to obtain the following result.
\begin{corollary}
\label{cor:extension-unique}
Let \(\tau\subset \overline{\mbox{K\"{a}h}(X)}\) be a regular cone of maximal dimension.
The GKZ system \(\mathcal{M}(A_{\mathrm{ext}},\beta)\) has at most one
power series solution of the form \(x^{\alpha}(1+g(z))\) with \(g(0)=0\) on \(X_{\tau}\).
Moreover, if this is a solution, then
\(\alpha=(\alpha_{i,j})\) with \(\alpha_{i,0}=-1/2\) for \(1\le i\le r\)
and \(\alpha_{i,j}=0\) for other \(i,j\).
\end{corollary}
Now we can prove our main result in this section.
\begin{theorem}
\label{thm:main-theorem-1}
For every toric resolution \(X_{G\Sigma'}\to X_{G\Sigma}\), there exists
at least one maximal degeneracy point in \(X_{G\Sigma'}\). 
\end{theorem}
\begin{proof}
Put \(\tau'=\overline{\mbox{K\"{a}h}}(X)\).
Then \(X_{\tau'}\) is a (possibly singular) affine toric variety. 
A smooth subdivision \(F\) of \(\tau'\) gives 
a toric resolution \(X_{F}\to X_{\tau'}\). Let \(\tau\) be a regular
maximal cone in \(F\).

Recall that \(\tau'\) determines the maximal triangulation \(\mathscr{T}_{\mathrm{max}}\). 
By definition, 
\begin{equation*}
\tau^{\vee}\supset \{\ell\in L_{\mathrm{ext}}\colon \ell_{i,j}\ge 0,~\nu_{i,j}
\notin \mathfrak{B}\}
\end{equation*}
for all bases \(\mathfrak{B}\in\mathscr{T}_{\mathrm{max}}\). Also, for 
all \(1\le i\le r\) and all \(\mathfrak{B}\in\mathscr{T}_{\mathrm{max}}\),
we have \(\nu_{i,0}\in\mathfrak{B}\).
It follows that 
the range \(\mathfrak{L}\) in the summation \eqref{eq:period-power-expansion} 
is contained in \(\tau^{\vee}\).
Consequently, for any \(\ell\in\mathfrak{L}\),
there exist uniquely non-negative integers \(m_{1},\ldots,m_{p-n}\)
such that 
\begin{equation*}
\ell = \sum_{k=1}^{p-n} m_{k} \ell^{(k)}.
\end{equation*}
As a function on \(X_{\tau}\), 
the normalized affine period integral \(\bar{\Pi}_{\gamma_{0}}\) becomes
\begin{equation}
\label{eq:period-extension}
\bar{\Pi}_{\gamma_{0}}(z) = 
(2\pi\sqrt{-1})^{n}\sum_{m\in \mathcal{S}} C_{\sum_{k=1}^{p-n} m_{k} \ell^{(k)}} z_{\tau}^{m},
\end{equation}
where \(\mathcal{S}=\left\{(m_{1},\ldots,m_{p-n})\in 
\mathbb{Z}^{p-n}_{\ge 0}\colon 
\ell_{i,0}\le 0~\mbox{for all}~i,~\mbox{where}~\ell=\sum_{k=1}^{p-n} m_{k}\ell^{(k)}\right\}\)
and \(C_{\ell}\) as well as \(\gamma_{0}\) are defined in 
\S\ref{subsection:affine-period-integrals}.

On one hand, from \eqref{eq:period-extension}, 
we see that \(\bar{\Pi}_{\gamma_{0}}\) extends holomorphically to 
the unique torus fixed point in \(X_{\tau}\). On the other hand, 
by Corollary \ref{cor:extension-unique}, there are no other 
normalized affine period integrals
with this property. This completes the proof.
\end{proof}

\section{Generalized Frobenius methods}

The aim of this section is to give a complete set of solutions 
to the GKZ hypergeometric system for our double covers via mirror symmetry. 
We will mainly follow the exposition in
\cite{1996-Hosono-Lian-Yau-gkz-generalized-hypergeometric-systems-in-mirror-symmetry-of-calabi-yau-hypersurfaces} and 
\cite{2006-Borisov-Horja-mellin-barnes-integrals-as-fourier-mukai-transforms}.
In what follows, let \(X=X_{\Sigma}\) be as in \S\ref{subsection:notation}.

\subsection{A series solution to GKZ systems}
We continuously assume 
the case \(\beta = \begin{bmatrix}\mathbf{0}&-1/2&\ldots&-1/2
\end{bmatrix}^{\intercal}\in \mathbb{Q}^{n+r}\). 
Let \(\alpha\in\mathbb{C}^{p+r}\) such that \(A_{\mathrm{ext}}(\alpha)=\beta\).
An obvious choice of \(\alpha\) is \(\alpha=(\alpha_{i,j})\) 
with \(\alpha_{i,j}=0\) for \(j\ne 0\) 
and \(\alpha_{i,0} = -1/2\) for \(i=1,\ldots,r\) (regarded as a column vector). 
A formal power series solution to the GKZ system 
\(\mathcal{M}(A_{\mathrm{ext}},\beta)\) 
is given by 
\begin{equation}
    \label{equation:gkz-system-solution-formal-series}
    \sum_{\ell\in L_{\mathrm{ext}}} 
    \frac{1}{\prod_{i=1}^{r}\prod_{j=0}^{n_i} \Gamma(\ell_{i,j}+\alpha_{i,j}+1)}
    x^{\ell+\alpha}.
\end{equation}
Notice that in the present case 
the formal power series \eqref{equation:gkz-system-solution-formal-series}
is non-zero and will be convergent around the origin if we choose the charge vectors
appropriately. However, in order to obtain an ``integral'' series, a renormalization is needed.
Following the treatment in 
\cite{1996-Hosono-Lian-Yau-gkz-generalized-hypergeometric-systems-in-mirror-symmetry-of-calabi-yau-hypersurfaces},
we multiply the series \eqref{equation:gkz-system-solution-formal-series}
by an overall constant factor \(\prod_{i=1}^{r}\Gamma(1+\alpha_{i,0})\).
Manipulating the identity \(\Gamma(z)\Gamma(1-z) = \pi\slash \sin(\pi z)\) 
\((z\notin\mathbb{Z})\),
we can rewrite the product \(\prod_{i=1}^{r}\Gamma(1+\alpha_{i,0})
\cdot\eqref{equation:gkz-system-solution-formal-series}\) into
the following form.
\begin{definition}
[The \(\Gamma\)-series, cf.~{\cite{1996-Hosono-Lian-Yau-gkz-generalized-hypergeometric-systems-in-mirror-symmetry-of-calabi-yau-hypersurfaces}*{Equation (3.5)}}]
Let
\begin{equation}
\label{equation:B-series-gamma-function}
\Phi^{\alpha}(x):=\sum_{\ell\in L_{\mathrm{ext}}} 
\frac{\prod_{i=1}^{r}\Gamma(-\ell_{i,0}-\alpha_{i,0})}
{\prod_{i=1}^{r}\Gamma(-\alpha_{i,0})
\prod_{i=1}^{r}\prod_{j=1}^{n_i} \Gamma(\ell_{i,j}+\alpha_{i,j}+1)}
(-1)^{\sum_{i} \ell_{i,0}}x^{\ell+\alpha}.
\end{equation}
\end{definition}

\begin{remark}
We can multiply \eqref{equation:B-series-gamma-function} 
by an overall factor \(\prod_{i=1}^{r}\prod_{j=1}^{n_{i}+1}\Gamma(\alpha_{i,j}+1)\) to get 
the usual product form.
It was pointed out in
\cite{2000-Hosono-local-mirror-symmetry-and-type-IIA-monodromy-of-calabi-yau-manifolds} 
that the Gamma function is crucial in order to get an integral, symplectic basis of 
the period integrals, although the period integrals 
obtained from the product form and the Gamma form are the same
up to a Gamma factor.
\end{remark}

\subsection{A cohomology-valued series associated with the holomorphic period}
Put \(D_{i,0}=-\sum_{j=1}^{n_{i}}D_{i,j}\) for all \(1\le i\le r\).
For each \(\ell\in L_{\mathrm{ext}}\), we define 
\begin{equation}
\label{eq:define-coh-valued-term}
\mathcal{O}^\alpha_\ell
:=\frac{\prod_{i=1}^{r}(-1)^{\ell_{i,0}}
\Gamma(-D_{i,0}-\ell_{i,0}-\alpha_{i,0})}{\prod_{i=1}^{r}\Gamma(-\alpha_{i,0})
\prod_{i=1}^r\prod_{j=1}^{n_i}\Gamma(D_{i,j}+\ell_{i,j}+\alpha_{i,j}+1)}.
\end{equation}
The quantity is understood as follows. The function \(1\slash \Gamma(z)\) is
an entire function on the complex plane. For \(j\ne 0\), we can expand
\begin{equation*}
\frac{1}{\Gamma(z+\ell_{i,j}+\alpha_{i,j}+1)}
\end{equation*}
into a power series in \(z\) around \(1\slash \Gamma(\ell_{i,j}+\alpha_{i,j}+1)\); namely
\begin{equation*}
\frac{1}{\Gamma(z+\ell_{i,j}+\alpha_{i,j}+1)} = 
1\slash \Gamma(\ell_{i,j}+\alpha_{i,j}+1) + a_{1} z + a_{2} z^{2} +\cdots.
\end{equation*}
Then for a divisor class \(D\in\mathrm{H}^{2}(X,\mathbb{Z})\), we define 
\begin{equation*}
\frac{1}{\Gamma(D+\ell_{i,j}+\alpha_{i,j}+1)} = 
\mathbf{1}\slash \Gamma(\ell_{i,j}+\alpha_{i,j}+1) + a_{1} D + a_{2} D^{2} +\cdots,
\end{equation*}
where \(\mathbf{1}\in\mathrm{H}^{0}(X,\mathbb{Z})\) is the Poincar\'{e}
dual of the fundamental class.
This is an honest element in \(\mathrm{H}^{\bullet}(X,\mathbb{C})\) since
\(D\) is nilpotent. 
For \(j=0\), we consider the \emph{deformed} coefficient
\begin{equation*}
\frac{\Gamma(-z-\ell_{i,0}-\alpha_{i,0})}{\Gamma(-\alpha_{i,0})}
\end{equation*}
and expand it into a power series in \(z\) 
around \(z=0\); namely
\begin{equation*}
\frac{\Gamma(-z-\ell_{i,0}-\alpha_{i,0})}{\Gamma(-\alpha_{i,0})}
=\frac{\Gamma(-\ell_{i,0}-\alpha_{i,0})}{\Gamma(-\alpha_{i,0})}+a_{1}z+a_{2}z^{2}+\cdots.
\end{equation*}
For any divisor class \(D\in\mathrm{H}^{2}(X,\mathbb{Z})\), we define 
\begin{equation*}
\frac{\Gamma(-D-\ell_{i,0}-\alpha_{i,0})}{\Gamma(-\alpha_{i,0})}
=\frac{\Gamma(-\ell_{i,0}-\alpha_{i,0})\cdot\mathbf{1}}{\Gamma(-\alpha_{i,0})}
+a_{1}D+a_{2}D^{2}+\cdots.
\end{equation*}
Consequently, 
\(\mathcal{O}_{\ell}^{\alpha}\) is a \emph{well-defined} 
element in \(\mathrm{H}^{\bullet}(X,\mathbb{C})\).

%\begin{remark}
%\label{remark:prod}
%For \(\alpha\in\mathbb{C}\) and \(c\le 0\), we have
%\begin{equation*}
%\frac{\Gamma(z+\alpha+1)}{\Gamma(z+c+\alpha+1)} = (-1)^{c}
%\frac{\Gamma(-z-c-\alpha)}{\Gamma(-z-\alpha)}=\prod_{k=0}^{-c-1} (z+\alpha-k).
%\end{equation*}
%\end{remark}

\begin{remark}
Note that \(1\slash \Gamma(w+D)\) is \emph{divisible} by \(D\)
if \(w\in\mathbb{Z}_{\le 0}\).
\end{remark}

The following lemma follows from the multiplicative 
property of the Gamma function.
\begin{lemma}
Let \(w\in\mathbb{C}\). Then for any \(D\in\mathrm{H}^{\bullet}(X,\mathbb{Z})\),
\begin{equation*}
\frac{(w+D)}{\Gamma(1+w+D)} = \frac{1}{\Gamma(w+D)}.
\end{equation*}
\end{lemma}
\begin{proof}
Fix \(w\in \mathbb{C}\), we have \(\Gamma(1+w+z)=(w+z)\Gamma(w+z)\)
as a function in \(z\). Therefore, 
\begin{equation*}
\frac{(w+z)}{\Gamma(1+w+z)} = \frac{1}{\Gamma(w+z)}.
\end{equation*}
\end{proof}

We now define the cohomology-valued series. Recall that we have an isomorphy
\(L_{\mathrm{ext}}\simeq L\) between the lattice relation of 
\(A_{\mathrm{ext}}\) and that of \(A\). The Mori cone \(\overline{\mathrm{NE}}(X)\)
can thus be regarded as a cone in \(L_{\mathrm{ext}}\)
which is also denoted by \(\overline{\mathrm{NE}}(X)\).
\begin{definition}
[cf.~\cites{1996-Hosono-Lian-Yau-gkz-generalized-hypergeometric-systems-in-mirror-symmetry-of-calabi-yau-hypersurfaces,2006-Borisov-Horja-mellin-barnes-integrals-as-fourier-mukai-transforms}]
\label{definition:b-series-definition}
We define the cohomology-valued \(B\) series to be
\begin{equation}
B_{X}^{\alpha}(x):=\left(\sum_{\ell\in \overline{\mathrm{NE}}(X)\cap L_{\mathrm{ext}}} 
\mathcal{O}_{\ell}^{\alpha} x^{\ell+\alpha}\right)
\exp\left(\sum_{i=1}^{r}\sum_{j=0}^{n_i}(\log x_{i,j}) D_{i,j}\right).
\end{equation}
\(B_{X}^{\alpha}\) is regarded as an element in 
\(\mathbb{C}\llbracket x_{i,j}\rrbracket
\otimes_{\mathbb{C}}\mathrm{H}^{\bullet}(X,\mathbb{C})\).
\end{definition}

\begin{proposition}
We have
\(\mathcal{O}_{\ell}^{\alpha}=0\) for \(\ell\in L_{\mathrm{ext}}
\setminus\overline{\mathrm{NE}}(X)\).
\end{proposition}

\begin{proof}
Let \(\ell=(\ell_{i,j})\in L_{\mathrm{ext}}\setminus \overline{\mathrm{NE}}(X)\). 
We claim that there exists a primitive collection 
\(\mathcal{P}\subset \{(i,j)\in J\colon \ell_{i,j}< 0\}\). 
Assuming the claim, we see that
\begin{equation*}
    \prod_{(i,j)\in\mathcal{P}} D_{i,j} 
\end{equation*}
appears in the numerator of \(\mathcal{O}^{\alpha}_{\ell}\) and hence 
\(\mathcal{O}_{\ell}^{\alpha}=0\) in \(\mathrm{H}^{\bullet}(X,\mathbb{C})\).

To prove the claim, we choose an ample divisor \(B\) with \(B.\ell<0\). 
\(B\) corresponds to a term order on \(\mathbb{C}[y_{i,j}]\), 
the homogeneous coordinate ring of \(X\). 
Write \(\ell=\ell^+-\ell^-\) as before. 
Then \(B.(\ell^+-\ell^-)<0\) and \(y^{\ell^-}\) will be the leading term 
of \(y^{\ell^+}-y^{\ell^-}\) with respect to \(B\).
Hence \(y^{\ell^-}\) is contained in the Stanley--Reisner ideal 
of \(X\). 

Using the fact that \(\alpha_{i,j}=0\) for all \(i\) and \(j\ne 0\),
we see that \(\mathcal{O}_{\ell}^{\alpha}\) is divisible by \(D_{i,j}\) for 
those \((i,j)\) such that \(\ell_{i,j}<0\) and hence it is divisible by 
\(\prod_{(i,j)\in\mathcal{P}}D_{i,j}\) for some
primitive collection \(\mathcal{P}\) of \(X\). This establishes the claim.
\end{proof}

This proposition allows us to rewrite
\begin{equation*}
B_{X}^{\alpha}(x)=\left(\sum_{\ell\in L_{\mathrm{ext}}} 
\mathcal{O}_{\ell}^{\alpha} x^{\ell+\alpha}\right)
\exp\left(\sum_{i=1}^{r}\sum_{j=0}^{n_i}(\log x_{i,j}) D_{i,j}\right).
\end{equation*}

\begin{proposition}[cf.~{\cite{2006-Borisov-Horja-mellin-barnes-integrals-as-fourier-mukai-transforms}*{Proposition 2.17}}]
\label{prop:coh-series-solu-gkz}
We regard \(B^{\alpha}_{X}(x)\) as an element in 
\(\mathbb{C}\llbracket x_{i,j}\rrbracket
\otimes_{\mathbb{C}}\mathrm{H}^{\bullet}(X,\mathbb{C})\).
For any \(h\in\mathrm{H}^{\bullet}(X,\mathbb{C})^{\vee}\),
the pairing \(\langle B^{\alpha}_{X}(x),h\rangle\in 
\mathbb{C}\llbracket x_{i,j}\rrbracket\)
is annihilated by \(\mathcal{M}(A_{\mathrm{ext}},\beta)\).
\end{proposition}

\begin{proof}
For simplicity, we drop the 
subscripts \(\alpha\) and \(X\) in \(B_{X}^{\alpha}(x)\).
For each variable \(x_{i,j}\), we have
\begin{align}
\label{equation:B-series-differentiate-x-i-j}
\begin{split}
\frac{\partial B(x)}{\partial x_{i,j}}&=\sum_{\ell\in L_{\mathrm{ext}}} 
\left(\frac{\ell_{i,j}+\alpha_{i,j}+D_{i,j}}{x_{i,j}}\right)
\mathcal{O}_{\ell}^{\alpha} x^{\ell+\alpha}
\exp\left(\sum_{i=1}^{r}\sum_{j=0}^{n_i}(\log x_{i,j}) D_{i,j}\right)\\
&=\left(\frac{1}{x_{i,j}}\right)
\sum_{\ell\in L_{\mathrm{ext}}}(\ell_{i,j}+\alpha_{i,j}+D_{i,j}) 
\mathcal{O}_{\ell}^{\alpha} x^{\ell+\alpha}\exp\left(\sum_{i=1}^{r}\sum_{j=0}^{n_i}
(\log x_{i,j}) D_{i,j}\right).
\end{split}
\end{align}
Hence the series \(\langle B^{\alpha}_{X}(x),h\rangle\)
is annihilated by the Euler operators.
Now we examine the box operators. 
We write
\begin{align}
\begin{split}
\eqref{equation:B-series-differentiate-x-i-j}=
\left(\frac{1}{x_{i,j}}\right)\sum_{\ell\in L_{\mathrm{ext}}}
\mathcal{O}_{\ell-\mathrm{e}_{i,j}}^{\alpha} x^{\ell+\alpha}
\exp\left(\sum_{i=1}^{r}\sum_{j=0}^{n_i}(\log x_{i,j}) D_{i,j}\right),
\end{split}
\end{align}
where \(\{\mathrm{e}_{i,j}\colon (i,j)\in J\}\) 
is the standard basis of \(\mathbb{Z}^{p+r}\). Here
we extend the definition of \(\mathcal{O}_{\xi}^{\alpha}\) to
any element \(\xi\in\mathbb{Z}^{p+r}\) by \eqref{eq:define-coh-valued-term}.

For \(l=l^{+}-l^{-}\in L_{\mathrm{ext}}\), we have
\begin{align}
\label{equation:B-series-1st}
\prod_{l_{i,j}>0}\left(\frac{\partial}{\partial x_{i,j}}\right)^{l_{i,j}}B(x)=
\sum_{\ell\in L_{\mathrm{ext}}}\mathcal{O}_{\ell-l^{+}}^{\alpha} 
x^{\ell+\alpha-l^{+}}\exp\left(\sum_{i=1}^{r}\sum_{j=0}^{n_i}(\log x_{i,j}) D_{i,j}\right).
\end{align}
Similarly, we have
\begin{align}
\label{equation:B-series-2st}
\prod_{l_{i,j}<0}\left(\frac{\partial}{\partial x_{i,j}}\right)^{-l_{i,j}}B(x)=
\sum_{\ell\in L_{\mathrm{ext}}}\mathcal{O}_{\ell-l^{-}}^{\alpha} 
x^{\ell+\alpha-l^{-}}\exp\left(\sum_{i=1}^{r}\sum_{j=0}^{n_i}(\log x_{i,j}) D_{i,j}\right).
\end{align}
Note that the ranges of the summations appeared 
on the right hand side of \eqref{equation:B-series-1st} and \eqref{equation:B-series-2st}
are the same.
Indeed, for any \(\ell\in L_{\mathrm{ext}}\) and \(l\in L_{\mathrm{ext}}\), 
there exists \(\ell'\in L_{\mathrm{ext}}\) such that
\(\ell-l^{+}=\ell'-l^{-}\) since \(l=l^{+}-l^{-}\in L_{\mathrm{ext}}\). 
This implies that \(\Box_{l} B(x)=0\).
\end{proof}

\begin{corollary}
\label{cor:full-set-soultion}
Assume \(X\) is smooth as before. When \(h\in\mathrm{H}^{\bullet}(X,\mathbb{C})^{\vee}\)
runs through a basis of \(\mathrm{H}^{\bullet}(X,\mathbb{C})^{\vee}\),
the series \(\langle B^{\alpha}_{X}(x),h\rangle\)
give a complete set of solution to \(\mathcal{M}(A_{\mathrm{ext}},\beta)\). 
\end{corollary}
\begin{proof}
It is clear that all the coefficients are linearly independent. On one hand,
for a general \(x\), we know that the solution space to 
\(\mathcal{M}(A_{\mathrm{ext}},\beta)\) 
has dimensional \(\mathrm{vol}_{r+n}(A_{\mathrm{ext}})\),
where \(\mathrm{vol}_{r+n}\) denotes the normalized volume in \(\mathbb{R}^{n+r}\).
On the other hand, by {\cite{2020-Hosono-Lee-Lian-Yau-mirror-symmetry-for-double-cover-calabi-yau-varieties}*{Proposition 1.2}}, 
\begin{equation*}
\mathrm{vol}_{r+n}(A_{\mathrm{ext}}) = \chi (X) = 
\dim \mathrm{H}^{\bullet} (X,\mathbb{C})
\end{equation*}
since \(X\) is a smooth toric variety.
\end{proof}

\begin{remark}
For odd \(n\), from the proof of 
\cite{2020-Hosono-Lee-Lian-Yau-mirror-symmetry-for-double-cover-calabi-yau-varieties}*{Theorem 2.2},
we know \(\chi(Y^{\vee})=\chi(X^{\vee})-\chi(X)\). Also from [loc.~cit., Theorem 2.1], we have
\begin{equation*}
\dim\mathrm{H}^{p,q}(Y^{\vee},\mathbb{C})=\dim\mathrm{H}^{p,q}(X^{\vee},\mathbb{C})~\mbox{for}~p+q\ne n.
\end{equation*}
Since \(X^{\vee}\) is also a smooth toric variety, it follows that 
\(\dim\mathrm{H}^{n}(Y^{\vee},\mathbb{C})=\chi(X)=\mathrm{vol}_{r+n}(A_{\mathrm{ext}})\). 
Together with Remark \ref{remark:period-cycle}, it suggests that 
the affine periods are all the solutions to the GKZ system.
\end{remark}

\bibliographystyle{amsxport}

\end{document}